\documentclass[a4paper]{amsart}
\usepackage{amsmath, amssymb, amsthm, url}
\usepackage{zref-clever}
\zcsetup{ lang = english, capfirst=true  }
\usepackage{braket}

\theoremstyle{plain}   
\newtheorem{thm}{Theorem}[section]
\newtheorem{theorem}[thm]{Theorem}

\newtheorem{lemma}[thm]{Lemma}
\newtheorem{cor}[thm]{Corollary}
\newtheorem{proposition}[thm]{Propsition}

\theoremstyle{remark}  
\newtheorem{remark}[thm]{Remark}

\newtheorem{example}[thm]{Example}

\theoremstyle{definition}  

\AddToHook{env/theorem/begin}{\zcsetup{countertype={thm=theorem}}}
\AddToHook{env/question/begin}{\zcsetup{countertype={thm=question}}}
\AddToHook{env/lemma/begin}{\zcsetup{countertype={thm=lemma}}}
\AddToHook{env/cor/begin}{\zcsetup{countertype={thm=cor}}}
\AddToHook{env/proposition/begin}{\zcsetup{countertype={thm=proposition}}}
\AddToHook{env/conjecture/begin}{\zcsetup{countertype={thm=conjecture}}}
\AddToHook{env/remark/begin}{\zcsetup{countertype={thm=remark}}}
\AddToHook{env/rem/begin}{\zcsetup{countertype={thm=rem}}}
\AddToHook{env/example/begin}{\zcsetup{countertype={thm=example}}}
\AddToHook{env/definition/begin}{\zcsetup{countertype={thm=definition}}}
\AddToHook{env/algorithm/begin}{\zcsetup{countertype={thm=algorithm}}}
\AddToHook{env/problem/begin}{\zcsetup{countertype={thm=problem}}}

\zcRefTypeSetup{thm}{
 Name-sg=Theorem,
 name-sg=Theorem,
 Name-pl=Theorems,
 name-pl=Theorems,
}
\zcRefTypeSetup{theorem}{
 Name-sg=Theorem,
 name-sg=Theorem,
 Name-pl=Theorems,
 name-pl=Theorems,
}
\zcRefTypeSetup{question}{
 Name-sg=Question,
 name-sg=Question,
 Name-pl=Questions,
 name-pl=Questions,
}
\zcRefTypeSetup{lemma}{
 Name-sg=Lemma,
 name-sg=Lemma,
 Name-pl=Lemmas,
 name-pl=Lemmas,
}
\zcRefTypeSetup{cor}{
 Name-sg=Corollary,
 name-sg=Corollary,
 Name-pl=Corollaries,
 name-pl=Corollaries,
}
\zcRefTypeSetup{proposition}{
 Name-sg=Proposition,
 name-sg=Proposition,
 Name-pl=Propositions,
 name-pl=Propositions,
}
\zcRefTypeSetup{conjecture}{
 Name-sg=Conjecture,
 name-sg=Conjecture,
 Name-pl=Conjectures,
 name-pl=Conjectures,
}
\zcRefTypeSetup{remark}{
 Name-sg=Remark,
 name-sg=Remark,
 Name-pl=Remarks,
 name-pl=Remarks,
}
\zcRefTypeSetup{rem}{
 Name-sg=Remark,
 name-sg=Remark,
 Name-pl=Remarks,
 name-pl=Remarks,
}
\zcRefTypeSetup{example}{
 Name-sg=Example,
 name-sg=Example,
 Name-pl=Examples,
 name-pl=Examples,
}
\zcRefTypeSetup{definition}{
 Name-sg=Definition,
 name-sg=Definition,
 Name-pl=Definitions,
 name-pl=Definitions,
}
\zcRefTypeSetup{algorithm}{
 Name-sg=Algorithm,
 name-sg=Algorithm,
 Name-pl=Algorithms,
 name-pl=Algorithms,
}
\zcRefTypeSetup{problem}{
 Name-sg=Problem,
 name-sg=Problem,
 Name-pl=Problems,
 name-pl=Problems,
}

\makeatletter
\DeclareFontEncoding{LS1}{}{}
\DeclareFontSubstitution{LS1}{stix}{m}{n}
\DeclareSymbolFont{symbols2}{LS1}{stixfrak}{m}{n}
\DeclareMathSymbol{\doubleplus}{\mathbin}{symbols2}{"E7}
\DeclareFontFamily{U}{mathx}{}
\DeclareFontShape{U}{mathx}{m}{n}{<-> mathx10}{}
\DeclareSymbolFont{mathx}{U}{mathx}{m}{n}
\DeclareMathAccent{\widecheck}{0}{mathx}{"71}
\makeatother

\makeatletter
\newcommand{\offset}[1][n]{o_{#1}}
\newcommand{\fib}[1]{f_{#1}}
\newcommand{\fibwordI}[1]{\omega_{#1}}
\newcommand{\fibwordR}[1]{\widetilde{\omega}_{#1}}
\newcommand{\fibwordX}{\@ifstar{\fibwordR}{\fibwordM}}
\newcommand{\fibword}{\@ifstar{\fibwordX}{\fibwordI}}
\newcommand{\cycfibword}[1]{\mathring\omega_{#1}}
\newcommand{\genP}[1][n]{\gamma_{#1}}
\newcommand{\genN}[1][n]{\gamma_{{#1}}^{\ast}}
\newcommand{\gen}{\@ifstar{\genN}{\genP}}
\newcommand{\boundP}[1][n]{a_{#1}}
\newcommand{\boundN}[1][n]{a_{{#1}}^{\ast}}
\newcommand{\bound}{\@ifstar{\boundN}{\boundP}}
\newcommand{\equvclassP}[2][n]{B_{#1}(#2)}
\newcommand{\equvclassN}[2][n]{B_{{#1}}^{\ast}(#2)}
\newcommand{\equivclass}{\@ifstar{\equvclassN}{\equvclassP}}
\newcommand{\equvclassLP}[2][n]{\widecheck{B}_{#1}(#2)}
\newcommand{\equvclassLN}[2][n]{\widecheck{B}_{{#1}}^{\ast}(#2)}
\newcommand{\equivclassL}{\@ifstar{\equvclassLN}{\equvclassLP}}
\newcommand{\equvclassRP}[2][n]{\widehat{B}_{#1}(#2)}
\newcommand{\equvclassRN}[2][n]{\widehat{B}_{{#1}}^{\ast}(#2)}
\newcommand{\equivclassR}{\@ifstar{\equvclassRN}{\equvclassRP}}
\newcommand{\numofclassN}[3][n]{m_{#1}(#2;#3)}
\newcommand{\numofclassP}[3][n]{m_{#1}(#2;#3)}
\newcommand{\numofclass}{\@ifstar{\numofclassN}{\numofclassP}}
\newcommand{\fibring}[1][n]{R_{#1}}
\newcommand{\fibringL}[1][n]{\widecheck R_{#1}}
\newcommand{\fibringR}[1][n]{\widehat R_{#1}}
\newcommand{\invL}[1][n]{\widecheck \rho_{#1}^{-1}}
\newcommand{\invR}[1][n]{\widehat \rho_{#1}^{-1}}
\newcommand{\bijLe}[1][n]{\widecheck \rho_{#1}}
\newcommand{\bijRe}[1][n]{\widehat \rho_{#1}}
\newcommand{\bijLx}[1][n]{\widecheck \varrho_{#1}}
\newcommand{\bijRx}[1][n]{\widehat \varrho_{#1}}
\newcommand{\bijL}{\@ifstar{\bijLx}{\bijLe}}
\newcommand{\bijR}{\@ifstar{\bijRx}{\bijRe}}
\newcommand{\subLP}[1][n]{\widecheck a_{#1}}
\newcommand{\subLN}[1][n]{\widecheck a_{#1}^{\ast}}
\newcommand{\subL}{\@ifstar{\subLN}{\subLP}}
\newcommand{\subRP}[1][n]{\widehat a_{#1}}
\newcommand{\subRN}[1][n]{\widehat a_{#1}^{\ast}}
\newcommand{\subR}{\@ifstar{\subRN}{\subRP}}
\newcommand{\sltozo}[1][n]{\nu_{#1}}
\newcommand{\concat}{\doubleplus}
\newcommand{\rev}[1]{{#1}^{\dagger}}
\newcommand{\pairI}[2]{P(#1,#2)}
\newcommand{\pairII}[3]{P(#1,#2,#3)}
\newcommand{\pair}{\@ifstar{\pairII}{\pairI}}
\newcommand{\lrfullP}[1][n]{w_{#1}}
\newcommand{\lrfullN}[1][n]{w_{#1}^{\ast}}
\newcommand{\lrfull}{\@ifstar{\lrfullN}{\lrfullP}}
\newcommand{\samecompP}[2][n]{\overset{#2}{\underset{#1}{\approx}}}
\newcommand{\samecompN}[2][n]{\overset{#2}{\underset{\overset{\smash\ast}{#1}}{\approx}}}
\newcommand{\samecomp}{\@ifstar{\samecompN}{\samecompP}}
\newcommand{\samestepP}[2][n]{\overset{#2}{\underset{#1}{\sim}}}
\newcommand{\samestepN}[2][n]{\overset{#2}{\underset{\overset{\smash\ast}{#1}}{\sim}}}
\newcommand{\samestep}{\@ifstar{\samestepN}{\samestepP}}
\newcommand{\nset}[1]{{[#1]}}
\makeatother

\newcommand{\ZZ}{\mathbb{Z}}

\begin{document}

\title{On gapped repeats in a cyclic Fibonacci word}
\author[T. Horiyama]{Takashi Horiyama}
\address[Horiyama]{Faculty of Information Science and Technology, Hokkaido University, Sapporo, Japan.}
\email{horiyama@ist.hokudai.ac.jp}
\author[Y. Numata]{Yasuhide Numata}
\address[Numata]{Faculty of Science, Hokkaido University, Sapporo, Japan.}
\email{nu@math.sci.hokudai.ac.jp}
\author[K. Seto]{Kazuhisa Seto}
\address[Seto]{Faculty of Information Science and Technology, Hokkaido University, Sapporo, Japan.}
\email{seto@ist.hokudai.ac.jp}
\author[S. Tsujie]{Shuhei Tsujie}
\address[Tsujie]{Department of Mathematics, Hokkaido University of Education, Asahikawa, Japan.}
\email{tsujie.shuhei@a.hokkyodai.ac.jp}

\thanks{%
  The first and third authors
  were partially supported by JSPS KAKENHI Grant Number
  JP22H03549.
  The second author
  was partially supported by JSPS KAKENHI Grant Number
  JP23K17298.
  The second and fourth authors
  were partially supported by JSPS KAKENHI Grant Number
  JP23H00081. 
}

\begin{abstract}
In this article, we consider the words with cyclic indices.
For given $s$, we consider the pair $(\iota,\kappa)$ of indices such that the word of length $s$ from $\iota$ is equal to the word of length $s$ from $\kappa$.
We give a characterization of such pairs for a cyclic Fibonacci word, and give the number of them.
\end{abstract}
\maketitle

\section{Introduction}
We call a map from $\nset{m}=\Set{1,\ldots,m}$ to $\Set{0,1}$ a \emph{word} of length $m$ over the alphabet $\Set{0,1}$.
We identify a word $x$ with the sequence $(x(1),\ldots,x(m))$.
Investigating the number of specific subwords in words is one of the central topics in combinatorics on words.
For words $x=(x(1),\ldots,x(m))$ and $x'=(x'(1),\ldots,x'(m'))$ of lengths $m$ and $m'$, respectively, let $x\concat x'$ denote their concatenation, i.e., the word $(x(1),\ldots,x(m),x'(1),\ldots,x'(m'))$ of length $m+m'$.
For a given word $x$, a subword of the form $v\concat v$ for some word $v$ is called a \emph{square} (or a \emph{tandem repeat}).
Squares are among the most fundamental structures in words.
Fraenkel and Simpson~\cite{FraenkelS98} showed that the number of distinct squares in a word of length $n$ is at most $2n$, and conjectured that the exact bound is $n$.
This upper bound was successively improved in Ilie--Rytter~\cite{Ilie07}, Deza--Franek--Thierry~\cite{DezaFT15}, and Thierry~\cite{thierry2020proof}, and the conjecture was finally proven by Brlek and Li~\cite{BrlekL25}.

A \emph{gapped repeat} is a natural generalization of a square.
It is a subword of the form $v\concat w \concat v$ for some words $v$ and $w$.
A gapped repeat $v\concat w \concat v$ is called an \emph{$\alpha$-gapped repeat} if $|v \concat w| \leq \alpha |v|$ for $\alpha \geq 1$, where $|x|$ denotes the length of a word $x$.

Equivalently, let $i$ and $k$ be the starting positions of the first and second occurrences of $v$ in $x$, respectively, and let $l$ be the length of $v$.
Then $v\concat w \concat v$ is an $\alpha$-gapped repeat if $(k-i)/l \leq \alpha$.
When $\alpha=1$, an $\alpha$-gapped repeat reduces to a square (i.e., the gap $w$ is empty).
Thus, gapped repeats generalize squares.
They are also fundamental structures in words and have been extensively studied.
For example, the upper and lower bounds on the number of maximal $\alpha$-gapped repeats in a word have been investigated.
Kolpakov--Podolskiy--Posypkin--Khrapov~\cite{MR3719915} and Kolpakov--Kucherov~\cite{MR1850247} showed that the number of maximal $\alpha$-gapped repeats in a word of length $m$ is $O(\alpha^2 m)$ and $\Omega(\alpha m)$, respectively.
Crochemore, Kolpakov, and Kucherov~\cite{MR3492485} improved the upper bound to $O(\alpha m)$.
For exact bounds, it was shown in Gawrychowski--I--Inenaga--K\"oppl--Manea~\cite{MR3742767} that the upper bound is $18\alpha m$, which was later improved to $3\left(\frac{\pi^2}{6}+\frac{5}{2}\right)\alpha m$ by I and K\"oppl~\cite{MR3906920}.

The number of squares and $\alpha$-gapped repeats has also been studied for specific families of words.
One of the most well-studied families is the sequence $\Set{\fibword{n}}{n\in\ZZ{>0}}$ of Fibonacci words, defined by $\fibword{0} = (0)$, $\fibword{1} = (1)$, and $\fibword{n} = \fibword{n-1}\concat\fibword{n-2}$.
Fibonacci words have many interesting properties (see, e.g.,~\cite{IliopoulosMS97,KishiNI23,KolpakovK99,Melancon00}), and the number of squares in the $n$-th Fibonacci word is exactly $2(f_{n-2}-1)$~\cite{FraenkelS99}, where $\fib{n}$ denotes the $n$-th Fibonacci number.
In Yamane--Nakashima--Seto--Horiyama~\cite{MR4872686}, upper and lower bounds on the number of maximal $\alpha$-gapped repeats in the $n$-th Fibonacci word were shown to be $2.1\alpha f_n + o(\alpha f_n)$ and $0.04\alpha f_n - o(\alpha f_n)$, respectively.

In this article, we consider the Fibonacci word $\cycfibword{n}$ with cyclic index set $\ZZ/\fib{n}\ZZ$.
For a given $s\in\ZZ_{>0}$, we study pairs $(\iota,\kappa)$ of indices such that the subword of length $s$ starting at position $\iota$ coincides with that starting at position $\kappa$.
We provide a characterization of such pairs and determine their number.
This article is organized as follows:
We will define notation and state our main results in 
\zcref{sec:mainresult}.
In \zcref{sec:proof}, we show the results.



\section{Notation and main results}
\label{sec:mainresult}
Here we define notation and state our main results,
which will be shown in \zcref{sec:proof}.

Let $I = \ZZ/m\ZZ$ and $w$ be a map from $I$ to $\Set{0,1}$.
We call $w$ a \emph{word with a cyclic index} of length $m$.
For $s\in \ZZ_{>0}$,
we define
\begin{align*}
\pair{w}{s}=
\Set{(\iota,\kappa)|t\in \nset{s}\implies
w(\iota+\overline{t-1})=w(\kappa+\overline{t-1})}
\subset I\times I,
\end{align*}
where $\nset{n}=\Set{1,\ldots,n}$.
For $\delta\in I$ and $s\in\ZZ_{>0}$,
we also define
\begin{align*}
\pair*{w}{s}{\delta}&=\Set{(\iota,\kappa)\in \pair{w}{s}|\iota-\kappa=\delta}.
\end{align*}
\begin{remark}
\label{lem:pair:minus}
By definition,
\begin{align*}
\pair*{w}{s}{-\delta}&=\Set{(\kappa,\iota)|(\iota,\kappa)\in \pair*{w}{s}{\delta}}.
\end{align*}
Hence
\begin{align*}
\# \pair*{w}{s}{-\delta}&=\# \pair*{w}{s}{\delta}.
\end{align*}
For $d<\frac{m}{2}$,
\begin{align*}
\# \Set{ \Set{\iota,\kappa}| (\iota,\kappa)\in\pair*{w}{s}{\overline{d}}\cup \pair*{w}{s}{\overline{-d}}}
&=\# \pair*{w}{s}{\overline{d}}.
\end{align*}
If $m$ is even and $d=\frac{m}{2}$,
\begin{align*}
\# \Set{ \Set{\iota,\kappa}| (\iota,\kappa)\in\pair*{w}{s}{\overline{d}}\cup \pair*{w}{s}{\overline{-d}}}
&=\frac{1}{2}\# \pair*{w}{s}{\overline{d}}.
\end{align*}
\end{remark}

We call a map from $\nset{m}$ to $\Set{0,1}$ a \emph{word} of length $m$.
We identify a word $x$ with a sequence $(x(1),\ldots,x(m))$.
For words
$\underline{x}=(x(1),\ldots,x(m))$ and $\underline{x'}=(x'(1),\ldots,x'(m'))$,
we define $\underline{x}\concat \underline{x'}$
to be
$(x(1),\ldots,x(m),x'(1),\ldots,x'(m'))$.
We define a \emph{Fibonacci word} $\fibword{n}$ by
\begin{align*}
\fibword{0} &=  (0),\\
\fibword{1} &=  (1),\\
\fibword{n} &=  \fibword{n-1}\concat\fibword{n-2}.
\end{align*}
The length of $\fibword{n}$ equals the Fibonacci number $\fib{n}$,
where
we define the $i$-th \emph{Fibonacci number} $\fib{i}$ by
\begin{align*}
\fib{0}&=1,\\
\fib{1}&=1,\\
\fib{i}&=\fib{i-1}+\fib{i-2}.
\end{align*}
Moreover
the number of zero's in $\fibword{n}$ equals $\fib{n-2}$,
the number of one's in $\fibword{n}$  equals $\fib{n-1}$.
Let $I_n$ to be $\ZZ/\fib{n}\ZZ$.
We define a map $\cycfibword{n}$ from $I_n$ to $\Set{0,1}$
by 
$\cycfibword{n}(\overline{i})=\fibword{n}(i)$ for $i\in \nset{\fib{n}}$.
We call $\cycfibword{n}$ the \emph{cyclic Fibonacci word}.
We are interested in $\pair*{\cycfibword{n}}{s}{\delta}$.
\begin{example}
Consider the cyclic Fibonacci word $\cycfibword{5}$.
Note that
\begin{align*}
\fibword{5}=(1,0,1,1,0,1,0,1).
\end{align*}
For $w=(w_1,\ldots,w_s)$,
we consider the set
\begin{align*}
J_{w}=\Set{\iota \in I_5|(\cycfibword{5}(\iota+\overline{0}),\cycfibword{5}(\iota+\overline{1}),\ldots,\cycfibword{5}(\iota+\overline{s-1}))=w}.
\end{align*}
For $s=1$,
we have
\begin{align*}
J_{(1)}&=\Set{\overline{1},\overline{3},\overline{4},\overline{6},\overline{8}},\\
J_{(0)}&=\Set{\overline{2},\overline{5},\overline{7}}.
\end{align*}
By definition,
$\pair{\cycfibword{5}}{1}=(J_{(1)})^2 \cup (J_{(0)})^2$.
Let $m(J)$ be the vector $(m_0,\ldots,m_7)$
such that $m_d$ is the number of the pairs $(\iota,\kappa)\in J^2$ with $\iota-\kappa=\overline{d}$.
By direct calculation,
we have
\begin{align*}
m(J_{(1)})&=(5,2,3,4,2,4,3,2)\\
m(J_{(0)})&=(3,0,1,2,0,2,1,0).
\end{align*}
Let $p_s$ be the vector $(p(0),\ldots,p(7))$
such that $p(d)=\#\pair*{\cycfibword{5}}{s}{\overline{d}}$. 
Then $p_s$ is the sum of $m(J)$.
Hence we have
\begin{align*}
p_1
&=(5,2,3,4,2,4,3,2)
 +(3,0,1,2,0,2,1,0)\\
&=(8,2,4,6,2,6,4,2).
\end{align*}
Let $p'_s = (p_s(0),p_s(5),p_s(2),p_s(7);p_s(4);p_s(1),p_s(6),p_s(3))$.
Then
\begin{align*}
p'_1=(8,6,4,2;2;2,4,6).
\end{align*}
For $s=2$,
we have
\begin{align*}
J_{(1,0)}&=\Set{\overline{1},\overline{4},\overline{6}},\\
J_{(1,1)}&=\Set{\overline{3},\overline{8}},\\
J_{(0,1)}&=J_{(0)}.
\end{align*}
By direct calculation we have
$m(J_{(1,1)})=(2,0,0,1,0,1,0,0)$.
Since
\begin{align*}
J_{(1,0)}=\Set{\iota-\overline{1} |\iota\in J_{(0)}},
\end{align*}
we have
$m(J_{(1,0)})=m(J_{(0)})=(3,0,1,2,0,2,1,0)$.
Hence
\begin{align*}
p_2
&=
2(3,0,1,2,0,2,1,0)
+(2,0,0,1,0,1,0,0)
\\
&=(8,0,2,5,0,5,2,0),\\
p'_2&=(8,5,2,0;0;0,2,5).
\end{align*}
For $s=3$,
$J_{(0,1)}$ splits into
\begin{align*}
J_{(0,1,0)}&=\Set{\overline{5}},\\
J_{(0,1,1)}&=\Set{\overline{2},\overline{7}}.
\end{align*}
We also have
$J_{(1,0,1)}=J_{(1,0)}$
and
$J_{(1,1,0)}=J_{(1,1)}$.
Since $J_{(0,1,0)}$ is a singleton,
we have $m(J_{(0,1,0)})=(1,0,0,0,0,0,0,0)$.
Since
\begin{align*}
J_{(0,1,1)}=\Set{\iota-\overline{1} | \iota\in J_{(1,1)}},
\end{align*}
we have $m(J_{(0,1,1)})=(2,0,0,1,0,1,0,0)$.
Hence
\begin{align*}
p_3
&=
  (3,0,1,2,0,2,1,0)
+2(2,0,0,1,0,1,0,0)
 +(1,0,0,0,0,0,0,0)
\\
&=(8,0,1,4,0,4,1,0),\\
p'_3&=(8,4,1,0;0;0,1,4).
\end{align*}
For $s=4$, $J_{(1,0,1,1)}$ splits into
\begin{align*}
J_{(1,0,1,1)}&=\Set{\overline{1},\overline{6}},\\
J_{(1,0,1,0)}&=\Set{\overline{4}}.
\end{align*}
We also have
$J_{(0,1,1,0)}=J_{(0,1,1)}$,
and
$J_{(1,1,0,1)}=J_{(1,1,0)}$.
We also have one more singleton $J_{(0,1,0,1)}=J_{(0,1,0)}$.
Hence
\begin{align*}
p_4
&=
3(2,0,0,1,0,1,0,0)
+2(1,0,0,0,0,0,0,0)
\\
&=(8,0,0,3,0,3,0,0),\\
p'_4&=(8,3,0,0;0;0,0,3).
\end{align*}
For $s=5$,
$J_{(1,1,0,1)}$ splits into the singletons
$J_{(1,1,0,1,0)}$ and $J_{(1,1,0,1,1)}$.
We also have
$J_{(0,1,1,0,1)}=J_{(0,1,1,0)}$,
$J_{(1,0,1,1,0)}=J_{(1,0,1,1)}$.
Also we have two more singletons.
Hence
\begin{align*}
p_5
&=
2(2,0,0,1,0,1,0,0)
+4(1,0,0,0,0,0,0,0)
\\
&=(8,0,0,2,0,2,0,0),\\
p'_5&=(8,2,0,0;0;0,0,2).
\end{align*}
For $s=6$,
$J_{(0,1,1,0,1)}$ splits into
the singletons
$J_{(0,1,1,0,1,0)}$
and
$J_{(0,1,1,0,1,1)}$.
We also have
$J_{(1,0,1,1,0,1)}=J_{(1,0,1,1,0)}$.
We have four more singletons.
Hence
\begin{align*}
p_6
&=
1(2,0,0,1,0,1,0,0)
+6(1,0,0,0,0,0,0,0)
\\
&=(8,0,0,1,0,1,0,0),\\
p'_6&=(8,1,0,0;0;0,0,1).
\end{align*}
For $s=7$,
$J_{(1,0,1,1,0,1)}$ splits into singletons
$J_{(1,0,1,1,0,1,0)}$
and 
$J_{(1,0,1,1,0,1,1)}$.
Since all are singletons, we have
\begin{align*}
p_7
&=
8(1,0,0,0,0,0,0,0)
\\
&=(8,0,0,0,0,0,0,0),\\
p'_7&=(8,0,0,0;0;0,0,0).
\end{align*}
\end{example}

Note that
\begin{align*}
\Set{1,2,\ldots,\fib{n}-1}
&=\coprod_{l=1}^{n-1} \Set{i | \fib{l}\leq i <\fib{l+1}}.
\end{align*}
Our main results are the following:
For the case where $s=1$, we have \zcref{mainthm:s=1}.
\begin{theorem}
\label{mainthm:s=1}
We have
\begin{align*}
\#\pair*{\cycfibword{n}}{1}{\overline{\fib{n-1}p}}
=
\begin{cases}
\fib{n}-2p&(0\leq p \leq \fib{n-2}),\\
\fib{n-3}&(\fib{n-2}\leq p \leq \fib{n-1}),\\
2p-\fib{n}&(\fib{n-1}\leq p < \fib{n}).
\end{cases}
\end{align*}
\end{theorem}
For the case where $2\leq s<\fib{n-1}$, we have \zcref{mainthm:s>1}:
\begin{theorem}
\label{mainthm:s>1}
For
$1<\fib{l}\leq s<\fib{l+1}<\fib{n}$,
\begin{align*}
&\#\pair*{\cycfibword{n}}{s}{\overline{\fib{n-1}p}}
=\\&
\begin{cases}
\fib{n}-p(s+1)
&
(0\leq p<\fib{n-l-1}),
\\
\fib{n-l-1}(\fib{l+1}-(s+1))+\fib{l}(\fib{n-l}-p)
&
(\fib{n-l-1}\leq p<\fib{n-l}),
\\
(\fib{l+1}-(s+1))(\fib{n-l+1}-p)
&
(\fib{n-l}\leq p<\fib{n-l+1}),
\\
0
&
(\fib{n-l+1}\leq p\leq\fib{n}-\fib{n-l+1}),
\\
(\fib{l+1}-(s+1))(\fib{n-l+1}-(\fib{n}-p))
&
(\fib{n}-\fib{n-l+1}\leq p \leq \fib{n}-\fib{n-l}),
\\
\fib{n-l-1}(\fib{l+1}-(s+1))+\fib{l}(\fib{n-l}-(\fib{n}-p))
&
(\fib{n}-\fib{n-l} < p\leq \fib{n}-\fib{n-l-1}),
\\
\fib{n}-(\fib{n}-p)(s+1)
&
(\fib{n}-\fib{n-l-1}< p< \fib{n}).
\end{cases}
\end{align*}
\end{theorem}
For the case where $\fib{n-1}\leq s<\fib{n}$, we have \zcref{mainthm:s>fn-1}.
\begin{theorem}
\label{mainthm:s>fn-1}
For
$1<\fib{n-1}\leq s<\fib{n}$,
\begin{align*}
\#\pair*{\cycfibword{n}}{s}{\overline{\fib{n-1}p}}
&=
\begin{cases}
\fib{n}
&(p=0),\\
\fib{n}-(s+1)
&(p=1),\\
0&(1<p<\fib{n-1}),\\
\fib{n}-(s+1)
&(p=\fib{n}-1).
\end{cases}
\end{align*}
\end{theorem}
Considering the case where $s=\fib{l+1}-1$,
we have the following as a corollary to \zcref{mainthm:s>1}:
\begin{cor}
\label{mainthm:cor:fib}
For $1<l\leq n-2$,
\begin{align*}
&\#\pair*{\cycfibword{n}}{\fib{l+1}-1}{\overline{\fib{n-1}p}}\\
&=
\begin{cases}
\fib{n}-p\fib{l+1}
&
(0\leq p <\fib{n-l-1}),
\\
(\fib{n-l}-p)\fib{l}
&
(\fib{n-l-1}\leq p<\fib{n-l}),
\\
0
&
(\fib{n-l}\leq p \leq \fib{n}-\fib{n-l}),
\\
(\fib{n-l}-(\fib{n}-p))\fib{l}
&
(\fib{n}-\fib{n-l}< p\leq \fib{n}-\fib{n-l}-1),
\\
\fib{n}-(\fib{n}-p)\fib{l+1}
&
(\fib{n}-\fib{n-l}-1< p <\fib{n})
.
\end{cases}
\end{align*}
\end{cor}

\section{Proof of main results}
\label{sec:proof}
Here we show our main results, i.e., \zcref{mainthm:s>1,mainthm:s=1,mainthm:s>fn-1}.

First we show some formula for Fibonacci words,
which will be used in the proof of \zcref{lem:num:firstsletters}.
For words
$\underline{x}=(x(1),x(2),\ldots,x(m))$, we define
\begin{align*}
\rev{\underline{x}}&=(x(m),x(m-1),\ldots,x(1)).
\end{align*}
Let 
\begin{align*}
\upsilon_n =
\begin{cases}
(1,0)&(\text{$n$ is odd}),\\
(0,1)&(\text{$n$ is even}).
\end{cases}
\end{align*}
Then we have the following:
\begin{lemma}
\label{lem:fibword:decomp:xx}
For $n>5$, 
$\fibword{n}=\fibword{n-2}\concat (\fibword{n-5}\concat\fibword{n-6}\concat\cdots \concat\fibword{1})\concat\upsilon_{n}\concat\fibword{n-2}$.
\end{lemma}
\begin{proof}
We show the equation by induction on $n$.
For $n=6,7$, we have
\begin{align*}
\fibword{6}
&=(1,0,1,1,0,1,0,1,1,0,1,1,0)\\
&=(1,0,1,1,0)\concat(1)\concat(0,1)\concat(1)\concat(1,0,1,1,0)\\
&=\fibword{4}\concat \fibword{1}\concat\upsilon_{6}\concat\fibword{4},\\
\fibword{7}
&=(1,0,1,1,0,1,0,1,1,0,1,1,0,1,0,1,1,0,1,0,1))\\
&=(1,0,1,1,0,1,0,1)\concat(1,0)\concat(1)\concat(1,0)\concat(1,0,1,1,0,1,0,1)\\
&=\fibword{5}\concat \fibword{2}\concat \fibword{1}\concat\upsilon_{7}\concat\fibword{5}.
\end{align*}
For $n>7$, 
we have
\begin{align*}
&\fibword{n-2}\concat (\fibword{n-5}\concat\fibword{n-6}\concat\fibword{n-7}\concat\cdots \concat\fibword{1}\concat\upsilon_{n})\concat\fibword{n-2}
\\
&=\fibword{n-2}\concat (\fibword{n-5}\concat\fibword{n-6})\concat(\fibword{n-7}\concat\cdots \concat\fibword{1}\concat\upsilon_{n})\concat\fibword{n-2}
\\
&=\fibword{n-2}\concat \fibword{n-4}\concat(\fibword{n-7}\concat\cdots \concat\fibword{1}\concat\upsilon_{n})\concat(\fibword{n-4}\concat\fibword{n-5}\concat\fibword{n-4})
\\
&=\fibword{n-2}\concat (\fibword{n-4}\concat\fibword{n-7}\concat\cdots \concat\fibword{1}\concat\upsilon_{n}\concat\fibword{n-4})\concat(\fibword{n-5}\concat\fibword{n-4}).
\end{align*}
By induction hypothesis,
\begin{align*}
&\fibword{n-2}\concat (\fibword{n-4}\concat\fibword{n-7}\concat\cdots \concat\fibword{1}\concat\upsilon_{n}\concat\fibword{n-4})\concat(\fibword{n-5}\concat\fibword{n-4})
\\
&=\fibword{n-2}\concat \fibword{n-2}\concat(\fibword{n-5}\concat\fibword{n-4})
\\
&=\fibword{n-2}\concat (\fibword{n-3}\concat\fibword{n-4})\concat(\fibword{n-5}\concat\fibword{n-4})
\\
&=(\fibword{n-2}\concat \fibword{n-3})\concat(\fibword{n-4}\concat\fibword{n-5})\concat\fibword{n-4}
\\
&=\fibword{n-1}\concat\fibword{n-3}\concat\fibword{n-4}
\\
&=\fibword{n-1}\concat\fibword{n-2}
\\
&=\fibword{n}.
\end{align*}
\end{proof}

\begin{lemma}
\label{lem:revfib:decomp:x}
For $n>4$,
$\rev{\fibword{n}}=\upsilon_n\concat\fibword{n-1}\concat (\fibword{n-4}\concat\fibword{n-5}\concat\cdots \concat\fibword{1})$.
\end{lemma}
\begin{proof}
For $n=5,6$,
\begin{align*}
\rev{\fibword{5}}&=(1,0,1,0,1,1,0,1)=\upsilon_5\concat\fibword{4}\concat\fibword{1},\\
\rev{\fibword{6}}&=(0,1,1,0,1,1,0,1,0,1,1,0,1)=\upsilon_6\concat\fibword{5}\concat\fibword{2}\concat\fibword{1}.
\end{align*}
Consider the case where $n>6$.
By induction hypothesis,
we have
\begin{align*}
&\rev{\fibword{n}}
\\
&=
\rev{\fibword{n-2}}\concat \rev{\fibword{n-1}}
\\
&=
(\upsilon_{n-2}\concat
\fibword{n-3}\concat \fibword{n-6}\concat\cdots \concat\fibword{1})
\concat
(\upsilon_{n-1}\concat
\fibword{n-2}\concat \fibword{n-5}\concat\cdots \concat\fibword{1})
\\
&=
\upsilon_{n-2}\concat
(\fibword{n-3}\concat \fibword{n-6}\concat\cdots \concat\fibword{1}
\concat
\upsilon_{n-1})
\concat
\fibword{n-2}\concat
(\fibword{n-5}\concat\cdots \concat\fibword{1})
\end{align*}
By definition we have $\upsilon_{n-2}=\upsilon_{n}$,
$\upsilon_{n-1}=\upsilon_{n-3}$
and 
$\fibword{n-2}=\fibword{n-3}\concat \fibword{n-4}$.
Hence
\begin{align*}
&\rev{\fibword{n}}
\\
&=\upsilon_{n-2}\concat
(\fibword{n-3}\concat \fibword{n-6}\concat\cdots \concat\fibword{1}
\concat
\upsilon_{n-1})
\concat
\fibword{n-2}\concat
(\fibword{n-5}\concat\cdots \concat\fibword{1})
\\
&=
\upsilon_{n}\concat
(\fibword{n-3}\concat \fibword{n-6}\concat\cdots \concat\fibword{1}
\concat
\upsilon_{n-3})
\concat
(\fibword{n-3}\concat\fibword{n-4})
\concat
(\fibword{n-5}\concat\cdots \concat\fibword{1})
\\
&=
\upsilon_{n}\concat
(\fibword{n-3}\concat \fibword{n-6}\concat\cdots \concat\fibword{1}
\concat
\upsilon_{n-3}
\concat
\fibword{n-3})\concat(\fibword{n-4}
\concat\fibword{n-5}\concat\cdots \concat\fibword{1}).
\end{align*}
Hence, by \zcref{lem:fibword:decomp:xx}, we have
\begin{align*}
&\rev{\fibword{n}}
\\
&=
\upsilon_{n}\concat
(\fibword{n-3}\concat \fibword{n-6}\concat\cdots \concat\fibword{1}
\concat
\upsilon_{n-3}
\concat
\fibword{n-3})\concat(\fibword{n-4}
\concat\fibword{n-5}\concat\cdots \concat\fibword{1})\\
&=
\upsilon_{n}\concat
\fibword{n-1}\concat(\fibword{n-4}
\concat\fibword{n-5}\concat\cdots \concat\fibword{1}).
\end{align*}
\end{proof}

\begin{cor}
\label{lem:w-:numzero}
Let $2\leq l\leq n-1$ and
$w=(\rev{\fibword{n}}(1),\ldots,\rev{\fibword{n}}(\fib{l}))$,
i.e., the first $\fib{l}$ letters in $\rev{\fibword{n}}$.
The number of zeros in $w$ is $\fib{l-2}$.
The number of ones in $w$ is $\fib{l-1}$.
\end{cor}
\begin{proof}
In the case where $n=3$, we have $\rev{\fibword{3}}=(1,0,1)$.
Since $w=(1,0)$ for $l=2$,
the number of zeros is $\fib{0}$ and the number of ones is $\fib{1}$.
In the case where $n=4$, we have $\rev{\fibword{4}}=(0,1,1,0,1)$.
Since $w=(0,1)$ for $l=2$,
the number of zeros is $\fib{0}$ and the number of ones is $\fib{1}$.
Since $w=(0,1,1)$ for $l=3$,
the number of zeros is $\fib{1}$ and the number of ones is $\fib{2}$.
In the case where $n>4$,
since the first $\fib{l}$ letters of $\fibword{n-1}$
is $\fibword{l}$,
it follows from
\zcref{lem:revfib:decomp:x}
that
$w$ is the first $\fib{l}$ letters of
the words $(0,1)\concat\fibword{l}$
or $w=(1,0)\concat\fibword{l}$
of length $\fib{l}+2$.
Since the final two letters of $\fibword{l}$ is $(1,0)$ or $(0,1)$,
the number of zeros (\emph{resp.} ones) in $w$ is the number of zeros (\emph{resp.} ones) in $\fibword{l}$, i.e., $\fib{l-2}$ (\emph{resp.} $\fib{l-1}$). 
\end{proof}

Next we give another definition of the cyclic Fibonacci words.
For $n\in \ZZ_{>0}$,
we define $\fibring$ to be  $\ZZ/\fib{n}\ZZ$.
Let
\begin{align*}
\offset =
\begin{cases}
-1 & (\text{$n$ is even}),\\
0 & (\text{$n$ is odd}).
\end{cases}
\end{align*}
We define the subsets $\fibringL$ and $\fibringR$ by
\begin{align*}
\fibringL&
=
\Set{\overline{x+\offset}\in\fibring|x\in \nset{\fib{n-2}}},
\\
\fibringR
&=
\Set{\overline{x+\fib{n-2}+\offset}\in\fibring|x\in \nset{\fib{n-1}}}.
\end{align*}
Let
\begin{align*}
\gen &= \overline{\fib{n-1}} \in \fibring,\\
\gen* &= \overline{\fib{n-2}} \in \fibring.
\end{align*}
We define maps $\bound$ and $\bound*$ from $I_n$ to $\fibring$ by
\begin{align*}
\bound(\iota)&=\gen \iota,\\
\bound*(\iota)&=\gen* \iota
\end{align*}
for $\iota\in I_n$.
Let $\sltozo$ be the map from $\fibring$ to $\Set{0,1}$ by
\begin{align*}
\sltozo(\alpha)
&=
\begin{cases}
0 & (\alpha\in\fibringL),\\
1 & (\alpha\in\fibringR).
\end{cases}
\end{align*}
We define maps $\lrfull$ and $\lrfull*$ from $I_n$ to $\Set{0,1}$ by
\begin{align*}
\lrfull(\iota) &
=
\sltozo(\bound(\iota)),\\
\lrfull*(\iota) &
=
\sltozo(\bound*(\iota))
\end{align*}
for $\iota\in I_n$.

The following are known:
\begin{proposition}
\label{lem:aisfibword}
For $n>1$, $\lrfull=\cycfibword{n}$.
\end{proposition}
Since we have $\lrfull=\cycfibword{n}$,
we also have the following:
\begin{cor}
\label{lem:lr-:fib}
For $n>1$,
\begin{align*}
(\lrfull*(\overline{0}),\lrfull*(\overline{1}),\ldots,\lrfull*(\overline{\fib{n}-1}))
=\rev{\fibword{n}}.
\end{align*}
\end{cor}
\begin{proof}
Since $\gen*=-\gen$, $\bound*(\overline{i})=\bound*(\overline{-i})=\bound*(\overline{\fib{n}-i})$.
Hence we have
\begin{align*}
(\bound*(\overline{0}),\bound*(\overline{1}),\ldots,\bound*(\overline{\fib{n}-1}))
&=
(\bound*(\overline{\fib{n}}),\bound*(\overline{\fib{n}-1}),\ldots,\bound*(\overline{1})),
\end{align*}
which implies
\begin{align*}
(\lrfull*(\overline{0}),\lrfull*(\overline{1}),\ldots,\lrfull*(\overline{\fib{n}-1}))
=\rev{(\lrfull(\overline{1}),\ldots,\lrfull(\overline{\fib{n}}))}.
\end{align*}
\end{proof}
Since we have \zcref{lem:lr-:fib},
we can count the numbers of indices
such that
$\bound*(\overline{t})\in \fibringL$
and 
$\bound*(\overline{t})\in \fibringR$.
\begin{lemma}
\label{lem:num:firstsletters}
If $1\leq \fib{l}-1\leq s \leq \fib{l+1}-1\leq\fib{n}$, then
\begin{align*}
\fib{l-2}&\leq \#\Set{t|0\leq t \leq s, \bound*(\overline{t})\in \fibringL} \leq \fib{l-1},\\
\fib{l-1}&\leq \#\Set{t|0\leq t \leq s, \bound*(\overline{t})\in \fibringR} \leq \fib{l}.
\end{align*}
\end{lemma}
\begin{proof}
By \zcref{lem:lr-:fib}, we have
\begin{align*}
\Set{t|0\leq t \leq s, \bound*(\overline{t})\in \fibringL}
&=
\Set{t|0\leq t \leq s, \sltozo(\bound*(\overline{t}))=0}
\\
&=
\Set{t|0\leq t \leq s, \rev{\fibword{n}}(t+1)=0}.
\end{align*}
By \zcref{lem:w-:numzero},
\begin{align*}
\#\Set{t|0\leq t < \fib{l}, \rev{\fibword{n}}(t+1)=0}=\fib{l-2},\\
\#\Set{t|0\leq t < \fib{l+1}, \rev{\fibword{n}}(t+1)=0}=\fib{l-1}.
\end{align*}
Hence we have 
\begin{align*}
\fib{l-2}&\leq \#\Set{t|0\leq t \leq s, \bound*(\overline{t})\in \fibringL} \leq \fib{l-1}.
\end{align*}
Similarly, by \zcref{lem:lr-:fib}, we have
\begin{align*}
\Set{t|0\leq t \leq s, \bound*(\overline{t})\in \fibringR}
&=
\Set{t|0\leq t \leq s, \rev{\fibword{n}}(t+1)=1}.
\end{align*}
By \zcref{lem:w-:numzero},
\begin{align*}
\#\Set{t|0\leq t < \fib{l}, \rev{\fibword{n}}(t+1)=0}=\fib{l-1},\\
\#\Set{t|0\leq t < \fib{l+1}, \rev{\fibword{n}}(t+1)=0}=\fib{l}.
\end{align*}
Hence we have
\begin{align*}
\fib{l-1}&\leq \#\Set{t|0\leq t \leq s, \bound*(\overline{t})\in \fibringR} \leq \fib{l}.
\end{align*}
\end{proof}

Next, we consider when the values $\lrfull(\iota)$ and $\lrfull(\kappa)$ are the same.
We define $\samestep*{t}$
to be the equivalence relation induced by
the classification
\begin{align*}
\Set{\fibringL+\bound*(\overline{t-1}),\fibringR+\bound*(\overline{t-1})},
\end{align*}
where
\begin{align*}
\fibringL+\bound*(\overline{t-1})
&=
\Set{\overline{a+\offset}\in\fibring|(t-1)\fib{n-2}<a\leq t\fib{n-2}},\\
\fibringR+\bound*(\overline{t-1})
&=\Set{\overline{a+\offset}\in\fibring|t\fib{n-2}<a\leq t\fib{n-2}+\fib{n-1}}.
\end{align*}
\begin{lemma}
For $\iota,\kappa\in I_n$ and $t\in\nset{\fib{n}}$, the following are equivalent:
\begin{enumerate}
\item $\lrfull(\iota+\overline{t-1})=\lrfull(\kappa+\overline{t-1})$.
\item $\bound(\iota)\samestep*{t}\bound(\kappa)$.
\end{enumerate}
\end{lemma}
\begin{proof}
By definition,
 $\lrfull(\iota+\overline{t-1})\neq\lrfull(\kappa+\overline{t-1})$
 means
 that 
 $\sltozo(\bound(\iota+\overline{t-1}))\neq\sltozo(\bound(\kappa+\overline{t-1}))$.
In the case where
 $\sltozo(\bound(\iota+\overline{t-1}))\neq\sltozo(\bound(\kappa+\overline{t-1})$, without loss of generality, we can assume that $\bound(\iota+\overline{t-1})\in\fibringR$
and
$\bound(\kappa+\overline{t-1})\in\fibringL$.
Since $\bound(\iota+\overline{t-1})=\bound(\iota)+\bound(\overline{t-1})$,
we have $\bound(\iota)\in \fibringL+\bound*(\overline{t-1})$.
Simalarly we have $\bound(\kappa)\in \fibringR+\bound*(\overline{t-1})$.
Hence
$\bound(\iota)$ and $\bound(\kappa)$
do not satisfy
$\bound(\iota)\samestep*{t}\bound(\kappa)$.

Conversely,
in the case where
$\bound(\iota)$ and $\bound(\kappa)$
do not satisfy
$\bound(\iota)\samestep*{t}\bound(\kappa)$, without loss of generality, we can assume that
 $\bound(\iota)\in \fibringL+\bound*(\overline{t-1})$
 and
$\bound(\kappa)\in \fibringR+\bound*(\overline{t-1})$.
Hence
we have $\bound(\iota)+\bound(\overline{t-1})\in \fibringL$,
which implies
$\sltozo(\bound(\iota+\overline{t-1}))=0$.
Similarly we have
$\sltozo(\bound(\kappa+\overline{t-1}))=1$.
Hence 
$\lrfull(\iota+\overline{t-1})\neq\lrfull(\kappa+\overline{t-1})$.
\end{proof}

Let
\begin{align*}
X=
\Set{\bound*(\overline{t})+\overline{\offset} |0\leq t \leq s}.
\end{align*}
For $a,b \in\ZZ$ with $a<b$ and $b-a<\fib{n}$, we define the relation $\sim$ by
\begin{align*}
a\sim b
\iff
\text{``$a\leq c < b\implies \overline{c}\not \in X$.''}
\end{align*}
We define
$\samecomp*{s}$
to be the equivalence relation on $\fibring$
induced by $\sim$.
\begin{remark}
Let $X=\Set{\bound*(\overline{t})+\overline{\offset} |0\leq t \leq s}$.
For $\alpha,\beta \in\fibring$, the following are equivalent:
\begin{enumerate}
\item $\alpha$ and $\beta$ do not satisfy $\alpha\samecomp*{s} \beta$.
\item The following hold:
\begin{enumerate}
\item There exist $a,b,c$ such that
$\alpha=\overline{a}$,
$\beta=\overline{b}$,
$\overline{c}\in X$,
$a\leq c < b$
and $b-a<\fib{n}$.
\item There exist $a',b',c'$ such that
$\alpha=\overline{a'}$,
$\beta=\overline{b'}$,
$\overline{c'}\in X$,
$b'\leq c' < a'$
and $a'-b'<\fib{n}$.
\end{enumerate}
\end{enumerate}
\end{remark}
\begin{lemma}
\label{lem:samecom:samestep}
Let $s\in\nset{\fib{n}}$.
For $\alpha,\beta \in \fibring$,
the following are equivalent:
\begin{enumerate}
\item $\alpha\samecomp*{s}\beta$.
\item For all $t\in\nset{s}$, $\alpha\samestep*{t}\beta$.
\end{enumerate}
\end{lemma}
\begin{proof}
Let
$X=\Set{\bound*(\overline{t})+\overline{\offset} | 0\leq t \leq s }$.

In the case where
$\alpha=\overline{a},\beta=\overline{b}\in \fibring$ and
$t\in\nset{s}$
do not satisfy $\alpha\samestep*{t}\beta$,
without loss of generality,
we can assume that
\begin{align*}
(t-1)\fib{n-2}+\offset< a \leq t\fib{n-2}+\offset < b \leq (t-1)\fib{n-2}+\fib{n}+\offset.
\end{align*}
Since $0\leq t-1 < t\leq s$,
$\overline{(t-1)\fib{n-2}}+\overline{\offset}$ and $\overline{t\fib{n-2}}+\overline{\offset}$ are in $X$.
Hence $\alpha,\beta$
do not satisfy
$\alpha\samecomp*{s}\beta$.

In the case where
$\alpha,\beta\in \fibring$
do not satisfy
$\alpha\samecomp*{s}\beta$,
there exisits $\overline{c},\overline{c'}\in X$ such that
$\alpha=\overline{a}$, $\beta=\overline{b}$ and
$c' < a<c\leq b \leq c'+\fib{n}$.
Without loss of generality, we can assume that
$\overline{c}=\bound*(\overline{t})+\overline{\offset}$,
$\overline{c'}=\bound*(\overline{t'})+\overline{\offset}$,
$0\leq t<t' \leq s$ and
\begin{align}
\label{lem:samecomp:pr:eq:a}
c'&< a\leq c < b \leq c'+\fib{n}.
\end{align}
If $t'-t=1$,
then $1\leq t'\leq s$.
Moreover
\ref{lem:samecomp:pr:eq:a}
means that
$\beta\in \fibringL+\bound*(\overline{t'-1})$ and
$\alpha\in \fibringR+\bound*(\overline{t'-1})$.
Hence
$\alpha$ and $\beta$ do not satisfy
$\alpha \samestep*{t'} \beta$.
Assume that $t'-t>1$.
In this case,
we have $0\leq t\leq s-2$ and $2 \leq t'\leq s$.
Hence we have
$1\leq t+1$, $t'$, $t'-1 \leq s$.
If
$c-\fib{n-1}< a$ and $b \leq c+\fib{n-2}$,
then we have
\begin{align*}
c-\fib{n-1} < a\leq c< b \leq c+\fib{n-2}.
\end{align*}
Since $\overline{c}=\bound*(\overline{t})$,
we have
$\beta\in \fibringL+\bound*(\overline{t})$ and
$\alpha\in \fibringR+\bound*(\overline{t})$.
Hence $\alpha$ and $\beta$ do not satisfy
$\alpha \samestep*{t+1} \beta$.
If
$a\leq c-\fib{n-1}$,
then we have
$a+\fib{n}\leq c-\fib{n-1}+\fib{n}=c+\fib{n-2}$,
Hence 
\begin{align*}
c< b \leq c'+\fib{n}< a+\fib{n}\leq c+\fib{n-2}.
\end{align*}
Since $(c'+\fib{n})-b<(c+\fib{n-2})-(c)=\fib{n-2}$,
we have $c'+\fib{n}-\fib{n-2}<b$.
Hence we have $c'+\fib{n}-\fib{n-2}<b\leq c'+\fib{n-2}$.
Since $\overline{c'+\fib{n}-\fib{n-2}}=\bound*(\overline{t'-1})$,
we have $\beta\in\fibringL+\bound*(\overline{t'-1})$.
On the other hand,
since $(a+\fib{n})-(c'+\fib{n}) <(c+\fib{n-2})-c=\fib{n-2}$,
we have
$(a+\fib{n})<(c'+\fib{n})+\fib{n-2}$.
Hence we have
$c'+\fib{n}< (a+\fib{n})\leq (c'+\fib{n})+\fib{n-1}$.
Since $c'+\fib{n}=(c'+\fib{n}-\fib{n-2})+\fib{n-2}$,
we have $\alpha\in\fibringR+\bound*(\overline{t'-1})$.
Hence
$\alpha$ and $\beta$ do not satisfy
$\alpha \samestep*{t'} \beta$.
Assume that
$c+\fib{n-2} < b$.
In this case, we have $c-\fib{n-1} < b-\fib{n}$.
Hence
\begin{align*}
c-\fib{n-1} < b-\fib{n}
\leq c' < a \leq c.
\end{align*}
If $c'-\fib{n-2}< b-\fib{n}$,
then $\beta\in \fibringL+\bound*(\overline{t'-1})$.
Moreover,
since $a-c'\leq c-(c-\fib{n-1})$,
$c' < a\leq c'+\fib{n-1}$.
Hence $\alpha\in \fibringR+\bound*(\overline{t'-1})$.
Hence
$\alpha$ and $\beta$ do not satisfy
$\alpha \samestep*{t'} \beta$.
Assume that
$b-\fib{n}\leq c'-\fib{n-2}$.
In this case, we consider $c'-2\fib{n-2}$.
Since
$c'-(b-\fib{n})<c-(c-\fib{n-1})=\fib{n-1}$,
$c'-2\fib{n-2}<b-\fib{n}$.
Hence $\beta\in \fibringL+\bound(\overline{t'-2})$.
On the other hand,
since $a\leq c$,
we have $a-\fib{n}\leq c-\fib{n}<c-\fib{n-1}$.
We also have $c-\fib{n-1}< b-\fib{n}\leq c'-\fib{n-2}$.
Hence we have $(c'-\fib{n-2})-(a-\fib{n})\geq (c-\fib{n-1})-(c-\fib{n})=\fib{n-2}$,
which implies $a-\fib{n}\leq (c'-\fib{n-2})-\fib{n-2}=c'-2\fib{n-2}$.
Hence
we have $c'-\fib{n-2}<a \leq c'-2\fib{n-2}+\fib{n}$,
which implies $\alpha\in\fibringR+\bound(\overline{t'-2})$.
Hence
$\alpha$ and $\beta$ do not satisfy
$\alpha \samestep*{t'-1} \beta$.
\end{proof}
By \zcref{lem:samecom:samestep},
we have the following:
\begin{cor}
\label{lem:pair:samecomp}
Let $s\in \nset{\fib{n}}$.
For $\iota,\kappa$, the following are equivalent:
\begin{enumerate}
\item
$(\iota,\kappa)\in\pair{\lrfull}{s}$.
\item
$\bound(\iota)\samecomp*{s}\bound(\kappa)$.
\end{enumerate}
\end{cor}

We define $\equivclass*{s}$ to be
\begin{align*}
\fibring/{\samecomp*{s}}
\end{align*}
i.e.,
the set of equvalent classes.

\begin{lemma}
\label{lem:p:countbycomp}
For $s\in\nset{\fib{n}}$ and $\delta\in I_n$,
\begin{align*}
\#\pair*{\lrfull}{s}{\delta}
&=
\sum_{C\in \equivclass*{s}}\#\Set{(\alpha,\alpha')\in C| \alpha-\alpha'=\bound(\delta)}.
\end{align*}
\end{lemma}
\begin{proof}
By \zcref{lem:pair:samecomp},
\begin{align*}
\pair{\lrfull}{s}
&=\Set{(\iota,\kappa)|\bound(\iota)\samecomp*{s}\bound(\kappa)}\\
&=\coprod_{C\in \equivclass*{s}}\Set{(\iota,\kappa)|\bound(\iota),\bound(\kappa)\in C}.
\end{align*}
Since $\bound$ is a bijection, we have
\begin{align*}
\pair*{\lrfull}{s}{\delta}
&=\coprod_{C\in \equivclass*{s}}\Set{(\iota,\kappa)|\bound(\iota),\bound(\kappa)\in C, \iota-\kappa=\delta}\\
&=\coprod_{C\in \equivclass*{s}}\Set{(\iota,\kappa)|\bound(\iota),\bound(\kappa)\in C, \bound(\iota-\kappa)=\bound(\delta)}\\
&=\coprod_{C\in \equivclass*{s}}\Set{(\iota,\kappa)|\bound(\iota),\bound(\kappa)\in C, \bound(\iota)-\bound(\kappa)=\bound(\delta)}
\end{align*}
for $\delta\in I_n$. Hence
\begin{align*}
\#\pair*{\lrfull}{s}{\delta}
&=
\sum_{C\in \equivclass*{s}}\#\Set{(\iota,\kappa)|\bound(\iota),\bound(\kappa)\in C, \bound(\iota)-\bound(\kappa)=\bound(\delta)}\\
&=
\sum_{C\in \equivclass*{s}}\#\Set{(\alpha,\alpha')\in C| \alpha-\alpha'=\bound(\delta)}.
\end{align*}
\end{proof}

\begin{lemma}
\label{lem:numofpairswithdiff}
Let $d\in\nset{\fib{n}}$, $o\in\ZZ$ and
$C=\Set{\overline{a+o}\in \fibring|a\in\nset{d}}$.
Consider 
$C(p)=\Set{(\alpha,\alpha')\in C^2|\alpha-\alpha'=\overline{p}}$
for $0\leq p <\fib{n}$.
If $d<\frac{\fib{n}}{2}$, then
\begin{align*}
\# C(p)
&=
\begin{cases}
d-p & (0\leq p \leq d),\\
0 & (d\leq p \leq \fib{n}-d),\\
d-(\fib{n}-p) & (\fib{n}-d\leq p < \fib{n}).
\end{cases}
\end{align*}
If $d\geq \frac{\fib{n}}{2}$, then
\begin{align*}
\# C(p)
&=
\begin{cases}
d-p & (0\leq p \leq \fib{n}-d),\\
2d-\fib{n} & (\fib{n}-d\leq p \leq d),\\
d-(\fib{n}-p) & (d\leq p < \fib{n}).
\end{cases}
\end{align*}
\end{lemma}
\begin{proof}
In the case where $p=0$,
we have $C(p)=\Set{(\alpha,\alpha)|\alpha\in C}$,
which implies $\# C(p)=d$.

Consider the case where $p>0$.
Let $a,a'\in\nset{\fib{n}}$ satisfy
$(\alpha,\alpha')\in C(p)$,
$\alpha=\overline{a}$ and 
$\alpha'=\overline{a'}$.
Then $a$ and $a'$ satisfy one of the following:
\begin{enumerate}
\item\label{lem:numofpairswithdiff:item:>}
$a>a'$ and $a-a'=p$; or
\item\label{lem:numofpairswithdiff:item:<}
$a<a'$ and $a-a'+\fib{n}=p$.
\end{enumerate}
Since $\overline{a}$ and $\overline{a'}$ are in $C$,
$a$ and $a'$ also satisfy $|a-a'|<d$.
Hence we have $d>a-a'$ and $a-a'>-d$.
Since $a-a'>-d$, we also have $a-a'+\fib{n}>-d+\fib{n}$.  
First consider the case where $d<\frac{\fib{n}}{2}$.
If
$p\leq \frac{\fib{n}}{2}$,
then 
we have $a-a'+\fib{n}>-d+\fib{n}>\frac{\fib{n}}{2}\geq p$,
which implies $a-a'+\fib{n}\neq p$.
Hence we condiser only the case \ref{lem:numofpairswithdiff:item:>}.
Since
\begin{align*}
C(p)=\Set{(\alpha+\overline{p},\alpha)|\alpha+\overline{p},\alpha\in C},
\end{align*}
we have $\#C(p)=d-p$ for $p < d$ and $\#C(p)=0$ for $p \geq d$.
If
$p\geq \frac{\fib{n}}{2}$,
then 
we have $a-a'<d<\frac{\fib{n}}{2}\leq p$,
which implies $a-a'\neq p$.
Hence we condiser only the case \ref{lem:numofpairswithdiff:item:<}.
Since
\begin{align*}
C(p)=\Set{(\alpha,\alpha+\overline{\fib{n}-p})|\alpha,\alpha+\overline{\fib{n}-p}\in C},
\end{align*}
we have $\#C(p)=d-(\fib{n}-p)$ for $d >\fib{n}-p$ and $\#C(p)=0$ for $d \leq\fib{n}-p$.
Next consider the case where $d\geq \frac{\fib{n}}{2}$.
If $p\geq d$, then
we have $a-a'<d\leq p$,
which implies $a-a'\neq p$
Hence we condiser only the case \ref{lem:numofpairswithdiff:item:<}.
Since $\fib{n}-p \leq \fib{n}-d \leq \frac{\fib{n}}{2}\leq d$,
we have $\#C(p)=d-(\fib{n}-p)$.
If $p\leq \fib{n}-d$, then
we have $a-a'+\fib{n}>-d+\fib{n}\geq p$,
which implies $a-a'+\fib{n}\neq p$.
Hence we condiser only the case \ref{lem:numofpairswithdiff:item:>}.
Since
$p\leq \fib{n}-d\leq \frac{\fib{n}}{2}\leq d$,
we have $\#C(p)=d-p$.
If $\fib{n}-d<p<d$, then
we consider the cases
\ref{lem:numofpairswithdiff:item:>}
and 
\ref{lem:numofpairswithdiff:item:<}.
Since $p<d$, we have
\begin{align*}
\#\Set{(\alpha+\overline{p},\alpha)|\alpha+\overline{p},\alpha\in C}=d-p.
\end{align*}
Since $d >\fib{n}-p$, we have
\begin{align*}
\#\Set{(\alpha,\alpha+\overline{\fib{n}-p})|\alpha,\alpha+\overline{\fib{n}-p}\in C}=d-(\fib{n}-p).
\end{align*}
Hence
$\#C(p)=(d-p)+(d-(\fib{n}-p))=2d-\fib{n}$.
\end{proof}

\begin{lemma}
\label{lem:numofpairs:s=1}
Let $\delta\in I_n$ and $0<p<\fib{n}$ satisfy
$\bound(\delta)=\overline{p}$.
We have
\begin{align*}
\#\pair*{\lrfull}{1}{\delta}
&
=
\begin{cases}
\fib{n}-2p&(0\leq p \leq \fib{n-2}),\\
\fib{n-3}&(\fib{n-2}\leq p \leq \fib{n-1}),\\
2p-\fib{n}&(\fib{n-1}\leq p < \fib{n}).
\end{cases}
\end{align*}
\end{lemma}
\begin{proof}
In the case where $s=1$,
$\equivclass*{s}$ 
consists of two equvalent classes $C_1$ and $C_2$,
which satisfy $\#C_1=\fib{n-1}>\frac{\fib{n}}{2}$ and $\#C_1=\fib{n-2}<\frac{\fib{n}}{2}$.
Hence it follows from \zcref{lem:numofpairswithdiff,lem:p:countbycomp}
that
\begin{align*}
&\#\pair*{\lrfull}{1}{\delta}
=\\&
\begin{cases}
(\fib{n-1}-p)+(\fib{n-2}-p)&(0\leq p \leq \fib{n-2}),\\
(2\fib{n-1}-\fib{n})+0&(\fib{n-2}\leq p \leq \fib{n-1}),\\
(\fib{n-1}-(\fib{n}-p))+(\fib{n-2}-(\fib{n}-p))&(\fib{n-1}\leq p < \fib{n}).
\end{cases}
\end{align*}
Hence
\begin{align*}
\#\pair*{\lrfull}{1}{\delta}
=
\begin{cases}
\fib{n}-2p&(0\leq p \leq \fib{n-2}),\\
\fib{n-3}&(\fib{n-2}\leq p \leq \fib{n-1}),\\
2p-\fib{n}&(\fib{n-1}\leq p < \fib{n}).
\end{cases}
\end{align*}
\end{proof}

By Cassini's identity,
$\overline{\fib{n-1}\fib{n+1}}=(\overline{-1})^n \in \fibring$.
Since $\overline{\fib{n+1}}=\overline{\fib{n-1}}\in \fibring$,
we have
$\bound(\overline{\fib{n-1}p})=\gen\overline{\fib{n-1}p}=\overline{\fib{n-1}\fib{n-1}p}=(\overline{-1})^n\overline{p}\in\fibring$.
Hence, by \zcref{lem:pair:minus},
we have \zcref{mainthm:s=1}.

Now we consider the case where $s>1$.
By using $\bound$ instead of $\bound*$,
we also define $\samecomp{s}$ and $\equivclass{s}$.
\begin{lemma}
\label{lem:numofequiveissame}
For $s\in \nset{\fib{n}}$ and $d\in \nset{\fib{n}}$,
\begin{align*}
\#\Set{C\in \equivclass*{s}|\#C=d}=\#\Set{C\in \equivclass{s}|\#C=d}.
\end{align*}
\end{lemma}
\begin{proof}
Let
\begin{align*}
X&=\Set{\bound*(\overline{t})+\overline{\offset}| 0\leq t\leq s},\\
X'&=\Set{\bound(\overline{t})+\overline{\offset}| 0\leq t\leq s}.
\end{align*}
Since $\bound*(\overline{t})=-\bound(\overline{t})$, we have
\begin{align*}
X'
&=\Set{-\bound*(\overline{t})+\overline{\offset}| 0\leq t\leq s}\\
&=\Set{-(x-\overline{\offset})+\overline{\offset}| \alpha\in X}.
\end{align*}
Since $X$ and $X'$ define $\samecomp*{s}$ and $\samecomp{s}$ respectively,
it follows that
\begin{align*}
\#\Set{C\in \equivclass*{s}|\#C=d}=\#\Set{C\in \equivclass{s}|\#C=d}.
\end{align*}
\end{proof}
We define
$\numofclass*{s}{d}$ to be
\begin{align*}
\#\Set{C\in \equivclass*{s}|\#C=d},
\end{align*}
i.e., 
the number of equivalence classes of size $d$.
By
\zcref{lem:numofequiveissame},
$\numofclass*{s}{d}$ also equals
$\#\Set{C\in \equivclass{s}|\#C=d}$.

\begin{example}
\label{ex:X:basecase:s=1}
Let $s=1$.
In this case,
\begin{align*}
X
&=\Set{\bound*(\overline{0})+\overline{\offset},\bound*(\overline{1})+\overline{\offset}}
=\Set{\overline{\offset+1},\overline{\fib{n-2}}+\overline{\offset+1}}.
\end{align*}
Hence $\equivclass*{s}=\Set{C_1,C_2}$, where
\begin{align*}
C_1&=\Set{\overline{i}+\overline{\offset}|i\in\nset{\fib{n-2}}},\\
C_2&=\Set{\overline{\fib{n-2}+i}+\overline{\offset}|i\in\nset{\fib{n-1}}}.
\end{align*}
Hence $\numofclass{s}{\fib{n-2}}=1$, $\numofclass{s}{\fib{n-1}}=1$,
and $\numofclass{s}{d}=0$ for $d\neq \fib{n-2}, \fib{n-1}$.
\end{example}

\begin{example}
\label{ex:X:basecase:s=2}
Let $s=2$.
In this case,
\begin{align*}
X&=\Set{\bound*(\overline{0})+\overline{\offset},\bound*(\overline{1})+\overline{\offset},\bound*(\overline{2})+\overline{\offset}}\\
&=\Set{\overline{0}+\overline{\offset},\overline{\fib{n-2}}+\overline{\offset},2\overline{\fib{n-2}}+\overline{\offset}}.
\end{align*}
Hence $\equivclass*{s}=\Set{C_1,C_2,C_3}$
where
\begin{align*}
C_1
&=\Set{\overline{i}+\overline{\offset}|i\in \nset{\fib{n-2}}},\\
C_2&=\Set{\overline{\fib{n-2}+i}+\overline{\offset}|i\in \nset{\fib{n-2}}},\\
C_3&=\Set{\overline{\fib{n-1}+\fib{n-4}+i}+\overline{\offset}|i\in \nset{\fib{n-3}}}.
\end{align*}
Hence $\numofclass{s}{\fib{n-2}}=2$, $\numofclass{s}{\fib{n-3}}=1$,
and $\numofclass{s}{d}=0$ for $d\neq \fib{n-2}, \fib{n-3}$.
\end{example}

\begin{remark}
\label{lem:size:lessthanhalf}
If $s>1$ and $n>2$, then
$\#C \leq \fib{n-2} < \frac{\fib{n}}{2}$
for $C\in \equivclass*{s}$.
\end{remark}

\begin{lemma}
\label{lem:numofpairs:s>1}
Let $\delta\in I_n$ and $0<p<\fib{n}$ satisfy
$\bound(\delta)=\overline{p}$.
For $1<s<\fib{n}$, we have
\begin{align*}
\#\pair*{\lrfull}{s}{\delta}
&=
\begin{cases}
\displaystyle
\sum_{d=p+1}^{\fib{n}-1}
\numofclass*{s}{d}\cdot (d-p)
&
(1\leq p < \frac{\fib{n}}{2}),
\\
\displaystyle
\sum_{d=\fib{n}-p+1}^{\fib{n}-1}
\numofclass*{s}{d}\cdot (d-(\fib{n}-p))
&
(\frac{\fib{n}}{2}\leq p <\fib{n}).
\end{cases}
\end{align*}
\end{lemma}
\begin{proof}
As in \zcref{lem:size:lessthanhalf},
$\#C< \frac{\fib{n}}{2}$ for $C\in \equivclass*{s}$
for $s>1$.
If $1\leq p \leq \frac{\fib{n}}{2}$,
then $\fib{n}-\# C > \frac{\fib{n}}{2}\geq p$.
Hence it follows from \zcref{lem:numofpairswithdiff,lem:p:countbycomp}
that
\begin{align*}
\#\pair*{\lrfull}{s}{\delta}
&=
\sum_{C\in \equivclass*{s}\colon \# C > p} \#C-p\\
&=
\sum_{d=p+1}^{\fib{n}-1} \numofclass*{s}{d}\cdot (d-p).
\end{align*}
If $\frac{\fib{n}}{2}\leq p < \fib{n}$,
then $\# C < \frac{\fib{n}}{2}\leq p$.
Hence it follows from \zcref{lem:numofpairswithdiff,lem:p:countbycomp}
that
\begin{align*}
\#\pair*{\lrfull}{s}{\delta}
&=
\sum_{C\in \equivclass*{s}\colon \# C > \fib{n}-p} \#C-(\fib{n}-p)\\
&=
\sum_{d=\fib{n}-p+1}^{\fib{n}-1} \numofclass*{s}{d}\cdot (d-(\fib{n}-p)).
\end{align*}
\end{proof}

Next, to calculate $\numofclass*{s}{d}$,
we consider the relation of $\bound*$, $\bound*[n-1]$ and $\bound*[n-2]$.
We define bijections $\bijL$ and $\bijR$
by 
\begin{align*}
\bijL\colon \fibring[n-2] &\to \fibringL\\
\overline{x+\offset[n-2]}&\mapsto \overline{x+\offset},
\intertext{where $x\in\nset{\fib{n-2}}$, and}
\bijR\colon \fibring[n-1] &\to \fibringR\\
\overline{x+\offset[n-1]}&\mapsto \overline{x+\fib{n-2}+\offset},
\end{align*}
where $x\in\nset{\fib{n-1}}$.
Then we have the following:
\begin{lemma}
\label{lem:inductionstep:explicit}
Let $\iota=\overline{i}, \iota'=\overline{i+1},\iota''=\overline{i+2},\iota'''=\overline{i+3} \in I_n$.
Let $a\in\nset{\fib{n}}$
satisfy $\overline{a}=\bound*(\iota)$.

\begin{enumerate}
\item
If
$0< a\leq \fib{n-3}$,
then we have the following:
\begin{enumerate}
\item
\label{lem:ll:0}
$\fib{n-2}< a+\fib{n-2}\leq\fib{n-1}$.
\item
\label{lem:ll:1}
$\bound*(\iota') +\overline{\offset}\in \fibringR$.
Hence $\bound*(\iota')+\overline{\offset} \not\in \fibringL$.
\item
\label{lem:ll:2}
$\bound*(\iota'')+\overline{\offset} \in \fibringR$.
Hence $\bound*(\iota'')+\overline{\offset} \not\in \fibringL$.
\item
\label{lem:ll:3}
$\bound*(\iota''')+\overline{\offset} \in \fibringL$.
\item
\label{lem:ll:f}
$\bound*(\iota''')+\overline{\offset}=\bijL(\invL(\bound*(\iota)+\overline{\offset})+\gen*[n-2])$.
\end{enumerate}

\item
If
$\fib{n-3}< a\leq \fib{n-2}$,
then we have the following:
\begin{enumerate}
\item
\label{lem:lr:0}
$\fib{n-1}< a+\fib{n-2} \leq \fib{n}$.
\item
\label{lem:lr:1}
$\bound*(\iota')+\overline{\offset} \in \fibringR$.
Hence $\bound*(\iota')+\overline{\offset} \not\in \fibringL$.
\item
\label{lem:lr:2}
$\bound*(\iota'')+\overline{\offset}\in\fibringL$.
\item
\label{lem:lr:f}
$\bound*(\iota'')+\overline{\offset}=\bijL(\invL(\bound*(\iota)+\overline{\offset})+\gen*[n-2])$.
\end{enumerate}

\item
If
$\fib{n-2}< a\leq \fib{n-1}$,
then we have the following:
\begin{enumerate}
\item
\label{lem:rl:0}
$\fib{n-1}< a+\fib{n-2} \leq \fib{n}$.
\item
\label{lem:rl:1}
$\bound*(\iota')+\overline{\offset} \in \fibringR$.
\item
\label{lem:rl:f}
$\bound*(\iota')+\overline{\offset}=\bijR(\invR(\bound*(\iota)+\overline{\offset})+\gen[n-1])$.
\end{enumerate}
\item
If
$\fib{n-1}< a \leq \fib{n}$,
then we have the following:
\begin{enumerate}
\item
\label{lem:rr:0}
$0 <a+\fib{n-2}-\fib{n}\leq \fib{n-2}$.
\item
\label{lem:rr:1}
$\bound*(\iota')+\overline{\offset} \in \fibringL$.
Hence 
$\bound*(\iota')+\overline{\offset} \not\in \fibringR$.
\item
\label{lem:rr:2}
$\bound*(\iota'')+\overline{\offset} \in \fibringR$.
\item
\label{lem:rr:f}
$\bound*(\iota'')+\overline{\offset}=\bijR(\invR(\bound*(\iota)+\overline{\offset})+\gen[n-1])$.
\end{enumerate}
\end{enumerate}
\end{lemma}

\begin{proof}
By direct calculation,
we have \zcref{lem:ll:0,lem:lr:0,lem:rr:0,lem:rl:0}.
\zcref{lem:ll:1,lem:ll:2,lem:ll:3,lem:lr:1,lem:lr:2,lem:rr:1,lem:rl:1,lem:rr:2}
follow from \zcref{lem:ll:0,lem:lr:0,lem:rr:0,lem:rl:0}.

First we consider \zcref{lem:ll:f}.
Since
\begin{align*}
\bound*(\iota)+\overline{\offset}=\overline{a+\offset}\in\fibringL,
\end{align*}
We have 
\begin{align*}
\invL(\bound*(\iota)+\overline{\offset})=\overline{a+\offset[n-2]}\in\fibring[n-2].
\end{align*}
Hence
\begin{align*}
\invL(\bound*(\iota)+\overline{\offset})+\gen*[n-2]=\overline{a+\fib{n-4}+\offset[n-2]}\in\fibring[n-2].
\end{align*}
Since $0< a \leq \fib{n-3}$,
we have $0< a+\fib{n-4}\leq \fib{n-2}$.
Hence
\begin{align*}
\bijL(\invL(\bound*(\iota)+\overline{\offset})+\gen*[n-2])=\overline{a+\fib{n-4}+\offset}\in\fibring.
\end{align*}
On the other hand,
\begin{align*}
\bound*(\iota''')+\overline{\offset}=\overline{a+3\fib{n-2}+\offset}\in\fibring.
\end{align*}
Since $3\fib{n-2}=\fib{n}+\fib{n-4}$,
\begin{align*}
\bound*(\iota''')+\overline{\offset}=\overline{a+\fib{n-4}+\offset}\in\fibring.
\end{align*}

Next we consider \zcref{lem:lr:f}.
Similar to \zcref{lem:ll:f},
we have
\begin{align*}
\invL(\bound*(\iota)+\overline{\offset})+\gen*[n-2]=\overline{a+\fib{n-4}+\offset[n-2]}=\overline{a-\fib{n-3}+\offset[n-2]}\in\fibring[n-2].
\end{align*}
Since $\fib{n-3} < a \leq\fib{n-2}$,
we have $0< a-\fib{n-3}\leq\fib{n-2}$.
Hence
\begin{align*}
\bijL(\invL(\bound*(\iota)+\overline{\offset})+\gen*[n-2])=\overline{a-\fib{n-3}+\offset}\in\fibring.
\end{align*}
On the other hand,
\begin{align*}
\bound*(\iota'')+\overline{\offset}=\overline{a+2\fib{n-2}+\offset}\in\fibring.
\end{align*}
Since $2\fib{n-2}=\fib{n-1}+\fib{n-4}$,
\begin{align*}
\bound*(\iota'')+\overline{\offset}
&=\overline{a+\fib{n-1}+\fib{n-4}+\offset}\\
&=\overline{a+\fib{n-1}+\fib{n-4}-\fib{n}+\offset}\\
&=\overline{a-\fib{n-3}+\offset}
\in\fibring.
\end{align*}

Next we consider \zcref{lem:rl:f}.
Since
\begin{align*}
\bound*(\iota)+\overline{\offset}=\overline{a+\offset}\in\fibringR,
\end{align*}
We have 
\begin{align*}
\invR(\bound*(\iota)+\overline{\offset})=\overline{a-\fib{n-2}+\offset[n-1]}\in\fibring[n-1].
\end{align*}
Hence
\begin{align*}
\invR(\bound*(\iota)+\overline{\offset})+\gen[n-1]
&=
\overline{a-\fib{n-2}+\fib{n-2}+\offset[n-1]}\\
&=\overline{a+\offset[n-1]}
\in\fibring[n-1].
\end{align*}
Since $\fib{n-2}< a \leq \fib{n-1}$,
we have
\begin{align*}
\bijR(\invR(\bound*(\iota)+\overline{\offset})+\gen[n-1])=\overline{a+\fib{n-2}+\offset}\in\fibring.
\end{align*}
On the other hand,
\begin{align*}
\bound*(\iota')+\overline{\offset}=\overline{a+\fib{n-2}+\offset}\in\fibring.
\end{align*}

Finally we consider \zcref{lem:rr:f}.
Similar to \zcref{lem:rl:f}.
\begin{align*}
\invR(\bound*(\iota)+\overline{\offset})+\gen[n-1]
&=\overline{a+\offset[n-1]}=\overline{a-\fib{n-1}+\offset[n-1]}
\in\fibring[n-1].
\end{align*}
Since $\fib{n-1} < a\leq \fib{n}$,
we have $0< a-\fib{n-1}\leq\fib{n-1}$.
Hence
\begin{align*}
\bijR(\invR(\bound*(\iota)+\overline{\offset})+\gen[n-1])&=\overline{a-\fib{n-1}+\fib{n-2}+\offset}\\
&=\overline{a-\fib{n-3}+\offset}\in\fibring.
\end{align*}
On the other hand,
\begin{align*}
\bound*(\iota'')+\overline{\offset}=\overline{a-\fib{n-3}+\offset}\in\fibring.
\end{align*}
\end{proof}

Now we define maps
$\subL*$ from $I_{n-2}$ to $\fibringL$
and
$\subR*$ from $I_{n-1}$ to $\fibringR$.
We define
$(\subL*(\overline{0}),\ldots,\subL*(\overline{\fib{n-2}-1}))$
and
$(\subR*(\overline{0}),\ldots,\subR*(\overline{\fib{n-1}-1}))$
to be the subsequence of
$(\bound*(\overline{0}),\ldots,\bound*(\overline{\fib{n}-1}))$
such that 
$\subL*(\iota)+\overline{\offset}\in \fibringL$ and
$\subR*(\iota)+\overline{\offset}\in \fibringR$,
respectively.
\begin{cor}
\label{cor:key:lem:indstep}
We have the following:
\begin{enumerate}
\item $\subL*(\iota)+\overline{\offset}=\bijL(\bound*[n-2](\iota)+\overline{\offset[n-2]})$
for $\iota\in I_{n-2}$.
\item $\subR*(\iota)+\overline{\offset}=\bijR(\bound[n-1](\iota)+\overline{\offset[n-1]})$
for $\iota\in I_{n-1}$.
\end{enumerate}
\end{cor}
\begin{proof}
Since
\begin{align*}
\bound*(\overline{0})+\overline{\offset}=\overline{0}+\overline{\offset}&\in\fibringR,
\end{align*}
we have $\subR*(\overline{0})=\overline{0}=\overline{\fib{n}}$.
Hence
\begin{align*}
\bijR(\bound[n-1](\overline{0})+\overline{\offset[n-1]})
&=
\bijR(\overline{0+\offset[n-1]})\\
&=
\bijR(\overline{\fib{n-1}+\offset[n-1]})\\
&=
\overline{\fib{n-1}+\fib{n-2}+\offset}\\
&=
\overline{\fib{n}+\offset}\\
&=
\subR*(\overline{0})+\overline{\offset}.
\end{align*}
Since
\begin{align*}
\bound*(\overline{1})+\overline{\offset}=\overline{\fib{n-2}}+\overline{\offset}&\in\fibringL,
\end{align*}
we have $\subL*(\overline{0})=\overline{\fib{n-2}}$.
Hence
\begin{align*}
\bijL(\bound[n-2](\overline{0})+\overline{\offset[n-2]})
&=
\bijL(\overline{0+\offset[n-2]})\\
&=
\bijL(\overline{\fib{n-2}+\offset[n-1]})\\
&=
\overline{\fib{n-2}+\offset}\\
&=
\subL*(\overline{0})+\overline{\offset}.
\end{align*}
Let $\subL*(\iota)=\bound*(\overline{i})=\overline{a}$.
Assume that
$\subL*(\iota)+\overline{\offset}=\bijL(\bound*[n-2](\iota)+\overline{\offset[n-2]})$.
If $0<a\leq\fib{n-3}$,
then $\subL*(\iota+\overline{1})=\bound*(\overline{i+3})$
by \zcref{lem:ll:1,lem:ll:2,lem:ll:3} of \zcref{lem:inductionstep:explicit}.
Moreover, by \zcref{lem:ll:f} of \zcref{lem:inductionstep:explicit},
we have
\begin{align*}
\subL*(\iota+\overline{1})+\overline{\offset}
&=
\bound*(\overline{i+3})+\overline{\offset}\\
&=
\bijL(\invL(\bound*(\iota)+\overline{\offset})+\gen*[n-2])\\
&=
\bijL(\bound*[n-2](\iota)+\overline{\offset[n-2]}+\gen*[n-2])\\
&=
\bijL(\bound*[n-2](\iota+\overline{1})+\overline{\offset[n-2]}).
\end{align*}
If $\fib{n-3}<a\leq\fib{n-2}$,
then $\subL*(\iota+\overline{1})=\bound*(\overline{i+2})$
by \zcref{lem:lr:1,lem:lr:2} of \zcref{lem:inductionstep:explicit}.
Moreover, by \zcref{lem:lr:f} of \zcref{lem:inductionstep:explicit},
we have
\begin{align*}
\subL*(\iota+\overline{1})+\overline{\offset}
&=
\bound*(\overline{i+2})+\overline{\offset}\\
&=
\bijL(\invL(\bound*(\iota)+\overline{\offset})+\gen*[n-2])\\
&=
\bijL(\bound*[n-2](\iota)+\overline{\offset[n-2]}+\gen*[n-2])\\
&=
\bijL(\bound*[n-2](\iota+\overline{1})+\overline{\offset[n-2]}).
\end{align*}

Let $\subR*(\iota)=\bound*(\overline{i})=\overline{a}$.
Assume that
$\subR*(\iota)+\overline{\offset}=\bijR(\bound[n-1](\iota)+\overline{\offset[n-1]})$.
If $\fib{n-2}<a\leq\fib{n-1}$,
then $\subR*(\iota+\overline{1})=\bound*(\overline{i+1})$
by \zcref{lem:rl:1} of \zcref{lem:inductionstep:explicit}.
Moreover, by \zcref{lem:rl:f} of \zcref{lem:inductionstep:explicit},
we have
\begin{align*}
\subR*(\iota+\overline{1})+\overline{\offset}
&=
\bound*(\overline{i+1})+\overline{\offset}\\
&=
\bijR(\invR(\bound(\iota)+\overline{\offset})+\gen[n-1])\\
&=
\bijR(\bound[n-1](\iota)+\overline{\offset[n-1]}+\gen[n-1])\\
&=
\bijR(\bound[n-1](\iota+\overline{1})+\overline{\offset[n-1]}).
\end{align*}
If $\fib{n-1}<a\leq\fib{n}$,
then $\subR*(\iota+\overline{1})=\bound*(\overline{i+2})$
by \zcref{lem:rr:1,lem:rr:2} of \zcref{lem:inductionstep:explicit}.
Moreover, by \zcref{lem:rr:f} of \zcref{lem:inductionstep:explicit},
we have
\begin{align*}
\subR*(\iota+\overline{1})+\overline{\offset}
&=
\bound*(\overline{i+2})+\overline{\offset}\\
&=
\bijR(\invL(\bound*(\iota)+\overline{\offset})+\gen[n-1])\\
&=
\bijR(\bound[n-1](\iota)+\overline{\offset[n-1]}+\gen[n-1])\\
&=
\bijR(\bound[n-1](\iota+\overline{1})+\overline{\offset[n-1]}).
\end{align*}
\end{proof}

Next we give explicit description of $\numofclass*{s}{d}$.
Let
\begin{align*}
\equivclassL*{s}&=\Set{C\in \equivclass*{s}| C\subset \fibringL},\\
\equivclassR*{s}&=\Set{C\in \equivclass*{s}| C\subset \fibringR}.
\end{align*}
\begin{lemma}
For $0<s<\fib{n}$,
$\equivclass*{s}=\equivclassL*{s}\cup\equivclassR*{s}$.
\end{lemma}
\begin{proof}
The relation $\samecomp*{s}$ is defined by the set
\begin{align*}
X
&=\Set{\bound*(\overline{t})+\overline{\offset}|0\leq t\leq s}.
\end{align*}
Since
$\bound*(\overline{0})+\overline{\offset}=\overline{0}+\overline{\offset}$
and
$\bound*(\overline{1})+\overline{\offset}=\overline{\fib{n-2}}+\overline{\offset}$
are in $X$,
$\alpha\in\fibringL$ and 
$\alpha'\in\fibringR$ do not satisfy
$\alpha\samecomp*{s}\alpha'$.
Hence
each equivalent class is a subset of $\fibringL$ or $\fibringR$.
\end{proof}

\begin{lemma}
\label{lem:inductionstep:b}
Let $2<\fib{l}\leq s <\fib{l+1}\leq \fib{n}$.
There exist
$\widecheck{s}$, $\widehat{s}$, $\bijL*$ and $\bijR*$
such that
\begin{enumerate}
\item $\fib{l-2}-1\leq \widecheck{s}< \fib{l-1}$,
\item $\fib{l-1}-1\leq \widehat{s}< \fib{l}$,
\item $\widecheck{s}+\widehat{s}=s-1$,
\item
$\bijL*\colon\equivclass*[n-2]{\widecheck{s}}\to\equivclassL*{s}$
is a bijection satisfying
$\#\bijL*(C)=\#C$ for each $C\in \equivclass*[n-2]{\widecheck{s}}$,
\item
$\bijR*\colon\equivclass[n-1]{\widehat{s}} \to\equivclassR*{s}$
is a bijection satisfying
$\#\bijR*(C)=\#C$ for $C\in \equivclass*[n-1]{\widehat{s}}$.
\end{enumerate}
\end{lemma}
\begin{proof}
Let
\begin{align*}
X
&=\Set{\bound*(\overline{t})+\overline{\offset}|0\leq t \leq s},
\end{align*}
Define $\widecheck{s}=\#X\cap \fibringL-1$
and
$\widehat{s}=\#X\cap \fibringR-1$.
Then $\widecheck{s}+\widehat{s}=(\#X\cap \fibringL-1)+(\widehat{s}=\#X\cap \fibringR-1)=\#X-2=s-1$.
Let
\begin{align*}
\widecheck{X}&=\Set{\bound*[n-2](\overline{i})+\overline{\offset[n-2]}|0\leq i\leq \widecheck{s}}\subset\fibring{n-2},\\
\widehat{X}&=\Set{\bound[n-1](\overline{i-1})+\overline{\offset[n-1]}|0\leq i\leq \widehat{s}}\subset\fibring{n-1}.
\end{align*}
Then $\widecheck{X}$ defines $\samecomp*[n-2]{\widecheck{s}}$ and
 $\widehat{X}$ defines $\samecomp[n-1]{\widehat{s}}$.
By \zcref{cor:key:lem:indstep},
we have have 
\begin{align*}
X\cap \fibringL
&=
\Set{\subL*(\overline{i})+\overline{\offset}|0\leq i \leq \widecheck{s}}
\\
&=
\Set{\bijL(\bound*[n-2](\overline{i})+\overline{\offset[n-2]})|0\leq i \leq\widecheck{s}}
\\
&=
\Set{\bijL(\alpha)|\alpha\in \widecheck X},
\\
X\cap \fibringR
&=
\Set{\subR*(\overline{i})+\overline{\offset}|0\leq i \leq \widehat{s}}\\
&=
\Set{\bijR(\bound[n-1](\overline{i})+\overline{\offset[n-1]})|0\leq i \leq\widehat{s}}\\
&=
\Set{\bijR(\alpha)|\alpha \in \widehat X}
.
\end{align*}
Hence $\bijL$ and $\bijR$ induce bijections
\begin{align*}
\bijL*\colon
\equivclass*[n-2]{\widecheck{s}}&\to\equivclassL*{s}\\
C&\mapsto \Set{\bijL(\alpha)| \alpha\in C},\\
\bijR*\colon
\equivclass[n-1]{\widehat{s}}  &\to\equivclassR*{s}\\
C&\mapsto \Set{\bijR(\alpha)| \alpha\in C}.
\end{align*}
Since $\bijL$ and $\bijR$ are bijective,
$\#\bijL*(C)=\#C$ and $\#\bijR*(C)=\#C$ for any $C$.

Now we consider $\widecheck{s}=\#X\cap \fibringL-1$
 and $\widecheck{s}=\#X\cap \fibringL-1$.
If $n$ is odd,
then $\offset=0$.
Hence
$\bound*(\overline{i})+\overline{\offset}=\bound*(\overline{i})$ and
\begin{align*}
X&=\Set{\bound*(\overline{t})+\overline{\offset}|0\leq t \leq s}\\
&=\Set{\bound*(\overline{t})|0\leq t \leq s}.
\end{align*}
By \zcref{lem:num:firstsletters},
we have $\fib{l-2}\leq\#X\cap\fibringL\leq \fib{l-1}$
and $\fib{l-1}\leq\#X\cap\fibringR\leq \fib{l}$.
These imply
$\fib{l-2}-1\leq \widecheck{s}< \fib{l-1}$
and
$\fib{l-1}-1\leq \widehat{s}< \fib{l}$.
If $n$ is even,
then $\offset=-1$.
Hence
$\bound*(\overline{i})+\overline{\offset}=\bound*(\overline{i})-\overline{1}$.
For $a\in\nset{\fib{n}}\setminus\Set{\fib{n-2},\fib{n}}$,
$\overline{a}\in\fibringL=\Set{\overline{0},\ldots,\overline{\fib{n-2}-1}}$ means $a\in \Set{1,\ldots,\fib{n-2}-1}$,
which implies $\overline{a+\offset}\in\fibringL$.
Similarly.
$\overline{a}\in\fibringR=\Set{\overline{\fib{n-2}},\ldots,\overline{\fib{n}-1}}$ means $a\in \Set{\fib{n-2}+1,\ldots,\fib{n-2}-1}$,
which implies $\overline{a+\offset}\in\fibringR$.
Hence
we have
\begin{align*}
\alpha\in\fibringL\iff\alpha+\overline{\offset}\in\fibringL,\\
\alpha\in\fibringR\iff\alpha+\overline{\offset}\in\fibringR
\end{align*}
for
\begin{align*}
\alpha\in \fibring\setminus\Set{\bound*(\overline{0})=\overline{0}=\overline{\fib{n}},\bound*(\overline{1})=\gen*=\overline{\fib{n-2}}}.
\end{align*}
Since $\bound*(\overline{0})\in \fibringL$, $\bound*(\overline{1})=\fibringR$,
$\bound*(\overline{0})+\overline{\offset}\in \fibringR$, $\bound*(\overline{1})+\overline{\offset}=\fibringL$,
we have
\begin{align*}
\#\Set{t|0\leq t \leq s, \bound*(\overline{t})+\overline{\offset}\in\fibringL}
&=\#\Set{t|0\leq t \leq s, \bound*(\overline{t})\in\fibringL},\\
\#\Set{t|0\leq t \leq s, \bound*(\overline{t})+\overline{\offset}\in\fibringR}
&=\#\Set{t|0\leq t \leq s, \bound*(\overline{t})\in\fibringR}
\end{align*}
for $s\geq 1$.
By \zcref{lem:num:firstsletters},
we have $\fib{l-2}\leq\#X\cap\fibringL\leq \fib{l-1}$
and $\fib{l-1}\leq\#X\cap\fibringR\leq \fib{l}$.
These imply
$\fib{l-2}-1\leq \widecheck{s}< \fib{l-1}$
and
$\fib{l-1}-1\leq \widehat{s}< \fib{l}$.
\end{proof}
\begin{proposition}
\label{lem:numofclass:s>0}
For $0<\fib{l} \leq s <\fib{l+1}<\fib{n}$,
\begin{align*}
\numofclass*{s}{d}&=
\begin{cases}
\fib{l+1}-(s+1)&(d=\fib{n-l+1}),\\
s+1-\fib{l-1}&(d=\fib{n-l}),\\
s+1-\fib{l}&(d=\fib{n-l-1}),\\
0&(\text{otherwise}).
\end{cases}
\end{align*}
\end{proposition}

\begin{proof}
We show the equations by induction on $s$.
The cases where $s=1$ with $l=1$ and $s=2$ with $l=2$
are the base cases,
which are in \zcref{ex:X:basecase:s=1,ex:X:basecase:s=2}.

By \zcref{lem:inductionstep:b}, 
We obtain bijections
\begin{align*}
\bijL*\colon\equivclass*[n-2]{\widecheck{s}}&\to\equivclassL*{s},\\
\bijL*\colon\equivclass*[n-1]{\widehat{s}}&\to\equivclassR*{s}
\end{align*}
with $\fib{l-2}-1\leq \widecheck{s}<\fib{l-1}$ and
 $\fib{l-1}-1\leq \widehat{s}<\fib{l}$.
Hence we have
\begin{align*}
\numofclass*{s}{d}
&=\numofclass*[n-2]{\widecheck{s}}{d}+\numofclass*[n-1]{\widehat{s}}{d}.
\end{align*}
Note that, for $s=\fib{l}-1$,
we have
\begin{align*}
&\begin{cases}
\fib{l+1}-(s+1)=\fib{l-1}&(d=\fib{n-l+1}),\\
s+1-\fib{l-1}=\fib{l-2}&(d=\fib{n-l}),\\
s+1-\fib{l}=0&(d=\fib{n-l-1}),\\
0&(\text{otherwise})
\end{cases}
\\
={}&
\begin{cases}
\fib{l}-(s+1)=0&(d=\fib{n-l+2)}),\\
s+1-\fib{l-2}=\fib{l-1}&(d=\fib{n-l+1}),\\
s+1-\fib{l-1}=\fib{l-2}&(d=\fib{n-l}),\\
0&(\text{otherwise}).
\end{cases}
\end{align*}
Hence,
by induction hypothesis,
we have
\begin{align*}
\numofclass*[n-2]{\widecheck{s}}{d}&=
\begin{cases}
\fib{l-1}-(\widecheck{s}+1)&(d=\fib{n-2-(l-3)}=\fib{n-l+1}),\\
\widecheck{s}+1-\fib{l-3}&(d=\fib{n-2-(l-2)}=\fib{n-l}),\\
\widecheck{s}+1-\fib{l-2}&(d=\fib{n-2-(l-1)}=\fib{n-l-1}),\\
0&(\text{otherwise}),
\end{cases}
\\
\numofclass*[n-1]{\widehat{s}}{d}&=
\begin{cases}
\fib{l}-(\widehat{s}+1)&(d=\fib{n-1-(l-2)}=\fib{n-l+1}),\\
\widehat{s}+1-\fib{l-2}&(d=\fib{n-1-(l-1)}=\fib{n-l}),\\
\widehat{s}+1-\fib{l-1}&(d=\fib{n-1-l}=\fib{n-l-1}),\\
0&(\text{otherwise}).
\end{cases}
\end{align*}
Hence, for $d=\fib{n-l+1}$,
\begin{align*}
\numofclass*{s}{d}
&=
\fib{l-1}-(\widecheck{s}+1)+\fib{l}-(\widehat{s}+1)=
\fib{l+1}-(s+1).
\end{align*}
For $d=\fib{n-l}$,
\begin{align*}
\numofclass*{s}{d}
&=
\widecheck{s}+1-\fib{l-3}+\widehat{s}+1-\fib{l-2}
=s+1-\fib{l-1}.
\end{align*}
For $d=\fib{n-l-1}$,
\begin{align*}
\numofclass*{s}{d}
&=
\widecheck{s}+1-\fib{l-2}+\widehat{s}+1-\fib{l-1}
=
s+1-\fib{k}.
\end{align*}
For $d\not\in\Set{\fib{n-l},\fib{n-l-1},\fib{n-l-2}}$,
we have
\begin{align*}
\numofclass*{s}{d}
&=0+0=0.
\end{align*}
\end{proof}

\begin{cor}
\label{lem:local:diff}
For $0<\fib{l} \leq s <\fib{l+1}<\fib{n}$,
we have
\begin{align*}
\Set{\# C |C\in \equivclass*{s}}
&=
\begin{cases}
\Set{\fib{n-l-1},\fib{n-l}}&(s=\fib{l+1}-1),\\
\Set{\fib{n-l-1},\fib{n-l},\fib{n-l+1}}&(s< \fib{l+1}-1).
\end{cases}
\end{align*}
\end{cor}
\begin{proof}
If $s=\fib{l+1}-1$, then $\fib{l+1}-(s+1)=0$.
Hence $\numofclass*{s}{\fib{n-l+1}}=0$.
\end{proof}

\begin{theorem}
\label{mainthm:lrseq}
Let $\delta\in I_n$ and $0<p<\fib{n}$
satisfy $\bound(\delta)=\overline{p}$.
For
$1<\fib{l}\leq s<\fib{l+1}<\fib{n}$,
\begin{align*}
&\#\pair*{\lrfull}{s}{\delta}
=\\&
\begin{cases}
\fib{n}-p(s+1)
&
(0< p<\fib{n-l-1}),
\\
\fib{n-l-1}(\fib{l+1}-(s+1))+\fib{l}(\fib{n-l}-p)
&
(\fib{n-l-1}\leq p<\fib{n-l}),
\\
(\fib{l+1}-(s+1))(\fib{n-l+1}-p)
&
(\fib{n-l}\leq p<\fib{n-l+1}),
\\
0
&
(\fib{n-l+1}\leq p\leq\fib{n}-\fib{n-l+1}),
\\
(\fib{l+1}-(s+1))(\fib{n-l+1}-(\fib{n}-p))
&
(\fib{n}-\fib{n-l+1}\leq p \leq \fib{n}-\fib{n-l}),
\\
\fib{n-l-1}(\fib{l+1}-(s+1))+\fib{l}(\fib{n-l}-(\fib{n}-p))
&
(\fib{n}-\fib{n-l} < p\leq \fib{n}-\fib{n-l-1}),
\\
\fib{n}-(\fib{n}-p)(s+1)
&
(\fib{n}-\fib{n-l-1}< p< \fib{n}).
\end{cases}
\end{align*}
\end{theorem}
\begin{proof}
By \zcref{lem:local:diff},
for
$d\not\in\Set{\fib{n-l-1},\fib{n-l},\fib{n-l+1}}$,
we have $\numofclass{s}{d}=0$.
Hence, by \zcref{lem:numofpairs:s>1},
we have $\#\pair*{\lrfull}{s}{\delta}=0$
for $\fib{n-l+1} \leq p \leq \frac{\fib{n}}{2}$.
Similarly, for $\frac{\fib{n}}{2}\leq p\leq\fib{n}-\fib{n-l+1}$,
we have $\#\pair*{\lrfull}{s}{\delta}=0$.
For $\fib{n-l} \leq p< \fib{n-l+1}$,
we have
\begin{align*}
\#\pair*{\lrfull}{s}{\delta}
&=(\fib{l+1}-(s+1))(\fib{n-l+1}-p).
\end{align*}
Similarly, for $\fib{n}-\fib{n-l+1}\leq p \leq \fib{n}-\fib{n-l}$,
we have
\begin{align*}
\#\pair*{\lrfull}{s}{\delta}
&=(\fib{l+1}-(s+1))(\fib{n-l+1}-(\fib{n}-p)).
\end{align*}
For
$\fib{n-l-1} \leq p <\fib{n-l}$,
we have 
\begin{align*}
\#\pair*{\lrfull}{s}{\delta}
&=(\fib{l+1}-(s+1))(\fib{n-l+1}-p)+(s+1-\fib{l-1})(\fib{n-l}-p)\\
&=(s+1-\fib{l-1})\fib{n-l}-(s+1-\fib{l-1})p\\
&\quad+(\fib{l+1}-(s+1))\fib{n-l+1}-(\fib{l+1}-(s+1))p\\
&=(s+1-\fib{l-1})\fib{n-l}+(\fib{l+1}-(s+1))\fib{n-l+1}-\fib{l}p\\
&=(s+1)\fib{n-l}-\fib{l-1}\fib{n-l}+\fib{l+1}\fib{n-l+1}-(s+1)\fib{n-l+1}-\fib{l}p\\
&=-(s+1)\fib{n-l-1}-\fib{l-1}\fib{n-l}+\fib{l+1}\fib{n-l+1}-\fib{l}p\\
&=-(s+1)\fib{n-l-1}-\fib{l-1}\fib{n-l}+\fib{l+1}\fib{n-l}+\fib{l+1}\fib{n-l-1}-\fib{l}p\\
&=(\fib{l+1}-(s+1))\fib{n-l-1}+(\fib{l+1}-\fib{l-1})\fib{n-l}-\fib{l}p\\
&=(\fib{l+1}-(s+1))\fib{n-l-1}+\fib{l}\fib{n-l}-\fib{l}p\\
&=(\fib{l+1}-(s+1))\fib{n-l-1}+\fib{l}(\fib{n-l}-p).
\end{align*}
Similarly, for $\fib{n}-\fib{n-(l)} < p\leq \fib{n}-\fib{n-(l+1)}$,
we have
\begin{align*}
\#\pair*{\lrfull}{s}{\delta}
\fib{n-l-1}(\fib{l+1}-(s+1))+\fib{l}(\fib{n-l}-(\fib{n}-p)).
\end{align*}
For $1\leq p < \fib{n-l-1}$,
we have 
\begin{align*}
\#\pair*{\lrfull}{s}{\delta}
&=
(\fib{l+1}-(s+1))(\fib{n-l+1}-p)\\
&\quad +(s+1-\fib{l-1})(\fib{n-l}-p)
+(s+1-\fib{l})(\fib{n-l-1}-p)\\
&=(\fib{l+1}-(s+1))\fib{n-l-1}+\fib{l}(\fib{n-l}-p)+(s+1-\fib{l})(\fib{n-l-1}-p)\\
&=(\fib{l+1}-(s+1)+(s+1-\fib{l}))\fib{n-l-1}+\fib{l}(\fib{n-l}-p)-(s+1-\fib{l})p\\
&=\fib{l-1}\fib{n-l-1}+\fib{l}(\fib{n-l}-p)-(s+1-\fib{l})p\\
&=\fib{l-1}\fib{n-l-1}+\fib{l}(\fib{n-l}-p+p)-(s+1)p\\
&=\fib{l-1}\fib{n-l-1}+\fib{l}\fib{n-l}-(s+1)p\\
&=\fib{n}-(s+1)p.
\end{align*}
Similarly, for $\fib{n}-\fib{n-l-1}< p< \fib{n}$,
we have
\begin{align*}
\#\pair*{\lrfull}{s}{\delta}=
\fib{n}-(\fib{n}-p)(s+1).
\end{align*}
\end{proof}

Considering the case where $s=\fib{l+1}-1$,
we have the following:
\begin{cor}
\label{mainthm:cor:lrseq}
Let $\delta\in I_n$ and $0<p<\fib{n}$
satisfy $\bound(\delta)=\overline{p}$.
For $1<l\leq n-2$,
\begin{align*}
&\#\pair*{\lrfull}{\fib{l+1}-1}{\delta}\\
&=
\begin{cases}
\fib{n}-p\fib{l+1}
&
(0<p <\fib{n-l-1}),
\\
(\fib{n-l}-p)\fib{l}
&
(\fib{n-l-1}\leq p<\fib{n-l}),
\\
0
&
(\fib{n-l}\leq p \leq \fib{n}-\fib{n-l}),
\\
(\fib{n-l}-(\fib{n}-p))\fib{l}
&
(\fib{n}-\fib{n-l}< p\leq \fib{n}-\fib{n-l}-1),
\\
\fib{n}-(\fib{n}-p)\fib{l+1}
&
(\fib{n}-\fib{n-l}-1< p <\fib{n})
.
\end{cases}
\end{align*}
\end{cor}

By Cassini's identity,
we have
$\bound(\overline{\fib{n-1}p})=(\overline{-1})^n\overline{p}\in\fibring$.
Hence 
\zcref{lem:pair:minus,lem:numofpairs:s>1,mainthm:lrseq,mainthm:cor:lrseq}
imply
\zcref{mainthm:s>1,mainthm:cor:fib}.

Finally we consider the case where
$\fib{n-1} \leq s <\fib{n}$.
\begin{proposition}
For $\fib{n-1} \leq s <\fib{n}$,
\begin{align*}
\numofclass*{s}{d}&=
\begin{cases}
\fib{n-1}-(s+1)&(d=2),\\
2(s+1)-\fib{n-1}&(d=1),\\
0&(\text{otherwise}).
\end{cases}
\end{align*}
\end{proposition}
\begin{proof}
By \zcref{lem:numofclass:s>0},
For $0<\fib{l} \leq s <\fib{l+1}<\fib{n}$,
\begin{align*}
\numofclass*{\fib{n-1}-1}{d}&=
\begin{cases}
\fib{n-1}-\fib{n-3}=\fib{n-2}&(d=\fib{2}),\\
\fib{n-1}-\fib{n-2}=\fib{n-3}&(d=\fib{1}),\\
0&(\text{otherwise}).
\end{cases}
\end{align*}
Hence, for each class $C$ in $\equivclass*{\fib{n-1}-1}$,
we have $\#C=2$ or $\#C=1$.
For $0<\fib{l} \leq s <\fib{l+1}<\fib{n}$,
$\bound*(s)$ is in some $C$ in $\equivclass*{\fib{n-1}-1}$
with $\#C=2$ and
 $C$ splits into two classes of size $1$ in $\equivclass*{s}$.
Hence we have the equation.
\end{proof}
\begin{theorem}
\label{mainthm:lrseq:s=fn-1}
Let $\delta\in I_n$ and $0<p<\fib{n}$
satisfy $\bound(\delta)=\overline{p}$.
For
$1<\fib{n-1}\leq s<\fib{n}$,
\begin{align*}
\#\pair*{\lrfull}{s}{\delta}
=
\begin{cases}
\fib{n}-(s+1)
&(p=1)\\
0&(1<p<\fib{n-1})\\
\fib{n}-(s+1)
&(p=\fib{n}-1).
\end{cases}
\end{align*}
\end{theorem}
\begin{proof}
By \zcref{lem:local:diff,mainthm:lrseq:s=fn-1},
we have the equation.
\end{proof}

By Cassini's identity,
we have
$\bound(\overline{\fib{n-1}p})=(\overline{-1})^n\overline{p}\in\fibring$.
Hence 
\zcref{mainthm:lrseq:s=fn-1}
implies
\zcref{mainthm:s>fn-1}.

\bibliographystyle{amsplain-url}
\bibliography{by-mr}

\end{document}